\documentclass[a4paper,10pt]{amsart}
\usepackage{amssymb, amsmath, graphicx, amsthm, enumitem, multicol,amsfonts}
\usepackage[margin=1in]{geometry}
\usepackage{cite}
\usepackage{ bbold }
\usepackage{longtable}
\usepackage{booktabs}
\usepackage{placeins}
\usepackage{hyperref}
\usepackage{tikz-cd}

\usepackage[utf8]{inputenc}

\DeclareMathOperator{\R}{\mathbb{R}}
\DeclareMathOperator{\C}{\mathbb{C}}
\DeclareMathOperator{\BC}{\mathbb{C}}

\DeclareMathOperator{\Z}{\mathbb{Z}}

\newcommand{\Hom}{\mathrm{Hom}}
\newcommand{\Aut}{\mathrm{Aut}}

\newcommand{\ba}{\begin{align}}
\newcommand{\ea}{\end{align}}
\newcommand{\bea}{\begin{eqnarray}}
\newcommand{\eea}{\end{eqnarray}}
\newcommand{\be}{\begin{equation}}
\newcommand{\ee}{\end{equation}}

\newcommand{\id}{\mathrm{id}}

\newcommand{\ie}{{\it i.e.~}}
\newcommand{\eg}{{\it e.g.~}}

\newcommand{\gl}{{\hat g}}
\newcommand{\hl}{{\hat h}}
\newcommand{\eL}{{\mathfrak e}} % elements of \hat L
\newcommand{\eD}{{e}} % vectors of defect representation \Omega

\newcommand{\E}{\mathrm{E}8}

\newcommand{\FA}{N}
\newcommand{\Zqp}{\Z_q \rtimes_\phi \Z_p}

\title{Non-Abelian Orbifolds of Lattice Vertex Operator Algebras}
\author{Thomas Gem\"unden}
\address{Thomas Gem\"unden, Department of Mathematics, ETH Zurich
	CH-8092 Zurich, Switzerland}
\email{thomas.gemuenden@math.ethz.ch}
\author{ Christoph A.~Keller}
\address{Christoph A. Keller,  Department of Mathematics, University of Arizona, Tucson, AZ 85721-0089, USA}
\email{christoph.keller@math.arizona.edu}

\theoremstyle{plain}% default
\newtheorem{thm}{Theorem}[section]
\newtheorem{lem}[thm]{Lemma}
\newtheorem{prop}[thm]{Proposition}
\newtheorem{cor}[thm]{Corollary}

\theoremstyle{definition}
\newtheorem{defn}{Definition}[section]

\theoremstyle{remark}

\newtheorem*{note}{Note}

\begin{document}

\begin{abstract}
We construct orbifolds of holomorphic lattice Vertex Operator Algebras for non-Abelian finite automorphism groups $G$. To this end, we construct twisted modules for automorphisms $g$ together with the projective representation of the centralizer of $g$ on the twisted module. This allows us to extract the irreducible modules of the fixed point VOA $V^G$, and to compute their characters and modular transformation properties. We then construct holomorphic VOAs by adjoining such modules to $V^G$. Applying these methods to extremal lattices in $d=48$ and $d=72$, we construct more than fifty new holomorphic VOAs of central charge 48 and 72, many of which have a very small number of light states.
\end{abstract}

\maketitle

\section{Introduction}
In this article we continue our exploration of holomorphic vertex operator algebras of central charge $c=48$ and $c=72$. On the one hand we are motivated by the fact that not too many examples of such VOAs have been constructed. On the other hand are motivated by the question of the existence of extremal VOAs \cite{MR1614941,MR2388095, Witten:2007kt}, and more generally by the physics of holography, which requires VOAs with `a small number of light states' \cite{Hartman:2014oaa}, that is a small number of vectors with low conformal weight. This means that for a VOA
\be
V = \bigoplus_n V_{(n)}\ ,
\ee
we want the $V_{(n)}$ to have a small dimension for the first few integers $n$. At $c=72$ for instance, an extremal VOA has the following `spectrum of light states':
\be
\dim V_{(1)} =0 \ , \qquad \dim V_{(2)}= 1\ , \qquad \dim V_{(3)}=1\ .
\ee
The dimensions of these spaces are as small as is compatible with the axioms of a VOA.
It is still an open question if such VOAs exist for central charge 48 or higher.
More generally it is difficult to construct VOAs with $\dim V_{(n)}$ small:
Most constructions give relatively large dimensions, or more precisely, dimensions that grow very quickly when the central charge increases. For lattice VOAs for instance, the dimensions grow at least polynomially in the rank of the lattice. We therefore need a way to decrease the dimension, for instance by projecting out a large number of vectors. In the present work, we achieve this by orbifolding.

Starting with a VOA $V$, the first step in orbifolding is to pick an automorphism group $G < \Aut(V)$, and restrict to the sub-VOA $V^G$ which is invariant under $G$.
By the combined results of \cite{MR1684904,MR3320313,2016arXiv160305645C}, if $V$ is tame (sometimes also called strongly rational) and holomorphic, and $G$ is a finite solvable group, then $V^G$ is again tame. It is however no longer holomorphic: it will in general have finitely many irreducible modules $M_i$, which can in fact be obtained from the twisted modules of $V$ \cite{MR1393116, 1997q.alg.....3016D,MR3715704}. The next step is thus to try to find a new holomorphic VOAs by considering holomorphic extensions of $V^G$:
that is, to adjoin an appropriate set of $V^G$-modules to $V^G$ to recover a holomorphic VOA given by $V^{orb(G)}= V^G \oplus_i M_i$. By the results of \cite{Evans:2018qgz}, this is possible only if a certain 3-cocycle in the modular tensor category of the modules of $V^G$ is trivial.
To construct the actual VOA structure on $V^{orb(G)}$, an understanding of intertwining operators for twisted modules is needed. This question was addressed for the case when the automorphisms of different twisted modules commute \cite{MR1338977}. More recently a lot of progress has been made in \cite{MR3709709, MR3715217}, which in particular gave the correct construction for the general, non-commuting case. In this work however we will content ourselves with computing the character of $V^{orb(G)}$.
See \cite{2016arXiv161109843M} for a nice overview on the status of orbifold VOAs.

Following this overall approach,
in \cite{Gemunden:2018mkh} we started with VOAs based on extremal lattices of dimensions $d=48$ and $72$ \cite{MR1662447,MR1489922,MR2999133,MR3225314}, and picked $G$ to be cyclic groups. In this way we constructed more than a hundred new examples of holomorphic VOAs with few light vectors. Although we did not find an extremal VOA, we found VOAs with much smaller dimensions than for unorbifolded lattice VOAs. Our construction of $V^{orb(G)}$ in these cases was based on  the theory of cyclic orbifolds, which  was completely worked out in \cite{vanEkeren2017}. In particular, we could avoid constructing the twisted modules of $V$ explicitly: their characters could be obtained from the known modular transformation of the characters of untwisted modules.

In this article we return to the same lattice VOAs, but now orbifold by non-cyclic and even non-Abelian automorphism groups $G$. In this way we construct about fifty new VOAs, some of which improve significantly on the previously obtained ones in that their dimensions are much lower. For $c=72$, as in \cite{Gemunden:2018mkh}, we can state one result of our constructions in the following form:
\begin{thm}\label{thmmain}
	There exists  a tame holomorphic VOA with central charge $c=72$ and
	$$
	\dim V_{(1)} =0 \ , \qquad \dim V_{(2)}= 12\ , \qquad \dim V_{(3)}=200\ .
	$$
\end{thm}
Note that the VOA in theorem~\ref{thmmain} improves over the VOA in the analog theorem in \cite{Gemunden:2018mkh}, for which the dimensions were larger. We prove in fact that this VOA cannot be obtained from an Abelian orbifold of a lattice VOA, by giving a lower bound on $\dim V_{(2)}$ for all such Abelian orbifolds; it is thus a genuinely non-Abelian construction. This indicates that the situation in $c=72$ is different from $c=24$, where all known 71 holomorphic VOAs can be constructed as cyclic lattice orbifolds \cite{MollerInPreparation}.

For $c=48$ on the other hand, we find several orbifold VOAs which give
$$
\dim V_{(1)} =0 \ , \qquad \dim V_{(2)}= 48\  .
$$
In \cite{Gemunden:2018mkh} we found cyclic orbifolds which have the same character. We thus found no non-cyclic orbifolds that improve over cyclic orbifolds. One might of course try to improve by considering other lattices and their orbifolds. However, \cite{MR3145153} suggests that other extremal lattices, if they exist, have smaller automorphism groups; \cite{MR3435733} gives a similar result for $\Gamma_{72}$. Using non-extremal lattices on the other hand introduces more light vectors coming from the short vectors of the lattice, which need to be eliminated by a powerful enough orbifold group.
This gives at least some evidence that $\dim V_{(2)}= 48$ may be the optimal lattice orbifold result for $c=48$, and that our $c=72$ constructions may be fairly close to optimal for lattice orbifolds. Here of course we leave completely open the possibility of an altogether different construction, which could lead to better results.

In the remainder of the introduction, let us describe our algorithm for constructing orbifolds of lattice VOAs in  more detail and point out the key ingredients.
To construct non-cyclic and non-Abelian orbifolds of lattice VOAs, we proceed in several steps, relying on various results. The first step is to find a suitable subgroup of $Aut(V_L)$ with which to orbifold. For this we start with a group of lattice automorphisms $G < Aut(L)$. The elements $g$ of $G$ then need to be lifted by a lift $\iota$ to $\iota(g) \in Aut(V_L)$ \cite{MR1745258}. In practice we find it useful to focus on cases where the lifted group $\langle\iota(G)\rangle$ is isomorphic to the original group $G$. We give an explicit construction for $\iota$.

The next step is to construct the $g$-twisted $V_L$-modules $W_g$. For this we use the results of \cite{MR820716} and \cite{MR2172171}. The central result here is that the full $g$-twisted module $W_g$ can be constructed as an induced representation from  a finite dimensional projective representation $\Omega_0$ of a certain finite Abelian group $N$.
Next we need to compute the action of the centralizer $C_g$ on $W_g$: for general reasons we know that there is a projective representation $\phi_g(\cdot)$ of $C_g$ on $W_g$ \cite{1997q.alg.....3016D, MR2023933,MR3715704}. To
obtain the twining characters $T(g,h;\tau)= \textrm{tr}_{W_g} \phi_g(h) q^{L_0-c/24}$ however, we need to construct $\phi_g$ explicitly.
Because we know how $\phi_g(\cdot)$ acts on the twisted vertex operators, it is enough to compute the restriction $\phi_g(\cdot)|_{\Omega_0}$ to the finite dimensional space $\Omega_0$. We achieve this by using Schur's lemma and the explicit representation of $N$ on $\Omega_0$. We then use $\phi_g(\cdot)$ to extract the irreducible $V^G$ modules from the twisted $V$-modules $W_g$ using \cite{1997q.alg.....3016D, MR3715704}, and compute their characters explicitly, which allows us to read off their modular transformation properties.

In the final step we extend $V^G$ to an holomorphic VOA $V^{orb(G)}$ by adjoining a suitable set of irreducible $V^G$-modules which we constructed in the previous step. This is not always possible:
we need to use the results of \cite{Evans:2018qgz} to check if there is an obstruction to this step, that is if in physics language the orbifold is `anomalous'. The central idea is that the irreducible modules of $V^G$ form a modular tensor category given by a Drinfeld double $D_\omega$ twisted by some 3-cocycle $\omega \in H^3(G,U(1))$.The result of \cite{Evans:2018qgz} then tells us that holomorphic extensions $V^{orb(G)}$ only exist if $\omega$ is trivial, and it also gives us the characters of all possible extensions. The condition $\omega$ trivial we can check from the modular transformation properties of the characters of the modules constructed above.

For many cases, it is not necessary to go through this whole procedure. If $G$ is cyclic, then one can of course simply use the general theory of cyclic orbifolds worked out in \cite{vanEkeren2017}, which simplifies things considerably: First of all, it is then unnecessary to construct the twisted modules and their twining characters, as it is enough to know the twining characters of the untwisted module, from which all other twining characters can be obtained by modular transformations \cite{vanEkeren2017,Gemunden:2018mkh}. Second, the condition $[\omega]=1$ reduces to a simple condition on the conformal weight of the twisted module, either called the type 0 condition or the level matching condition.

What we observe in this article is that in conjunction with \cite{Evans:2018qgz}, the cyclic methods of \cite{vanEkeren2017} can be applied not just to cyclic orbifold groups, but to a much larger class of non-Abelian groups: We define $G$ to be an \emph{effectively cyclic} orbifold group if for every pair $(a \in G,b\in C_a)$, we can find a cyclic subgroup $C<G$ such that $a,b\in C$. For such a $G$ we can compute all twisted twining characters by considering orbifolds by various cyclic groups $C$. A large class of effectively cyclic groups that we orbifold with are of the form
\be
G = \Z_q \rtimes_\phi \Z_p\ .
\ee
For these groups we find in practice that the condition $[\omega]=1$ can be checked by testing the type 0 condition for cyclic subgroups of $G$.
For effectively cyclic groups we therefore do not need to construct the twisted modules and its twining characters, and often it is also much easier to check for holomorphic extensions.

There are of course groups that are not effectively cyclic, for which we need our full algorithm: the simplest example is $\Z_2 \times \Z_2$, where the pair $(g_1,g_2)$ can not be embedded in a cyclic group.
We illustrate this example and other not effectively cyclic examples for the smallest holomorphic VOA, namely the $E_8$ lattice VOA. Any orbifold of this VOA is of course highly constrained, since any holomorphic extension has to return the original $E_8$ VOA, which is the only holomorphic VOA of that central charge. It is therefore not at all surprising that only 14 subgroups of $Aut(L_{E8})$ have a standard lift that gives trivial $\omega$, all of which are Abelian, and most of which are cyclic. We construct the orbifold for the smallest non-cyclic such group $\Z_3\times \Z_3$. We note that it has discrete torsion, since $H^2(\Z_3\times \Z_3,U(1))$ is not trivial. We find however that the modular orbit whose contribution is controlled by the discrete torsion vanishes. This is of course necessary, since we every choice of discrete torsion must return the original $E_8$ VOA, which is exactly what we find.
To illustrate our methods, we also construct the twisted twining characters for some groups with obstructions, namely an $S_3$ and a $\Z_2\times\Z_2$, and check explicitly that the modular transformations of their characters agree with the expected modular data.

This article is organized in the following way. In section~\ref{s:defs}, we briefly discuss the general theory of orbifolds and their extensions, and how it applies to cyclic orbifolds. In section~\ref{s:CC} we introduce the notion of an effectively cyclic orbifold, and discuss classes of examples of the form $\Z_q \rtimes_\phi \Z_p$. In section~\ref{s:lifting} we discuss the lifting of lattice automorphisms to VOA automorphisms. We give a lifting algorithm that works for a large class of cases, and discuss the construction of a splitting map for the lifting. In section~\ref{s:twistedmodules} we work out the construction of twisted modules of lattice VOAs, focusing on the construction of the defect representation. In section~\ref{s:chars} we  construct the projective representation $\phi_g$ of $C_g$, which allows us to construct the twining characters of all twisted modules.
In section~\ref{s:E8} we apply these methods to some orbifolds of the $E_8$ lattice VOA.
In section~\ref{s:lattice} we finally apply our methods to our main interest: constructing new holomorphic VOAs for $c=48$ and $72$. We construct a large number of effectively cyclic but non-Abelian orbifolds of extremal lattice VOAs. In particular we construct the VOAs advertised in theorem~\ref{thmmain}.

\emph{Acknowledgments:}
We thank Terry Gannon, Sven M\"oller and Nils Scheithauer for useful discussions. We thank Gerald H\"ohn and Yi-Zhi Huang for helpful discussions and comments on the draft. TG thanks the Department of Mathematics at University of Arizona for hospitality.
The work of TG is supported by the Swiss National Science Foundation Project Grant 175494.

\section{Twisted Modules, Orbifolds and extensions}\label{s:defs}
\subsection{Twisted modules}
We want to orbifold a tame holomorphic VOA $V$ by some finite subgroup $G<Aut(V)$ of automorphisms.
We call a VOA $V$ \emph{tame} if it is rational, $C_2$-cofinite, simple, self-contragredient and of CFT-type.
The significance of this is that \cite{MR2387861} established that the fusion rules for the modules of a tame VOA satisfy the Verlinde formula \cite{Verlinde:1988sn} and hence that the modules
form a modular tensor category \cite{MR2468370}.
By the combined results of \cite{MR1684904,MR3320313,2016arXiv160305645C}, if $G$ is a finite solvable group and $V$ is tame, then $V^G$ is again tame. This is believed to hold even if $G$ is finite but not solvable; in the present work however $G$ will always be solvable.
In general $V^G$ will not be holomorphic. To construct new holomorphic VOAs $V^{orb(G)}$, we therefore need to find holomorphic extensions of $V^G$ by adjoining a suitable set of its modules. To this end it is necessary to introduce the notion of twisted modules.

Given a VOA $V$ and an automorphism $g\in \Aut(V)$ of order $N$, a  $g$-twisted $V$-module $W_g$ consists of
a $\C$-graded vector space $W_g = \bigoplus_{\lambda \in \C}W_{\lambda}$
together with a linear map $Y_W: V \to (End(W))[[x^{1/N},x^{-1/N}]]$ satisfying various axioms which can be found in \cite{1997q.alg.....3016D}. In the following, the important one is \emph{twisting compatibility}, which states that
\be
Y_W(v,x) = \sum_{k \in -r/N+\Z}v^W_k x^{-k-1}\ \textrm{if}\ gv = \exp(2\pi i r/N)v\ , \qquad r=0,\ldots N-1\ .
\ee
There is then of course also the usual notion of an irreducible twisted module.
We will denote by $\rho_g$ the conformal weight of $W_g$.

Next we consider the action of a commuting automorphism $h$ on the $g$ twisted module.
Let $V$ be a $C_2$-cofinite VOA, $g$ and $h$ automorphisms of $V$ such that $gh = hg$ and $(W_g,Y_W(\cdot,x))$ an irreducible $g$-twisted $V$-module.
Then by the results of \cite{1997q.alg.....3016D, MR2023933,MR3715704}
there exists a  linear map $\phi_g(h):W_g \to W_g$ such that
\be\label{eq:phi2}
\phi_g(h)Y_W(v,x)\phi_g(h)^{-1} = Y_W(h v,x)\ ,
\ee
which is unique up to multiplication by a scalar. $\phi_g(\cdot)$ is a projective representation of $C_g$ with some cocycle $c_g(\cdot, \cdot)$.
We will need to  work out this representation explicitly in order to compute the expression for the twisted twining characters
\be
T(g,h;\tau)= \textrm{tr}_{W_g} \phi_g(h) q^{L_0-c/24}
\ee

The projective representation $\phi_g$ is useful to obtain the irreducible modules of the fixed point VOA $V^G$.
From \cite{1997q.alg.....3016D} we know that $W_g$ decomposes into modules  for $\C_{c_g}[C_g]\otimes V^G$ as $C_g$ through a Schur-Weyl duality
\be
W_g = \bigoplus_{\chi \in \textrm{Irr}_{c_g}(C_g)} V_{[g,\chi]} \otimes W_{[g,\chi]},
\ee
where $\textrm{Irr}_{c_g}$ denotes the set of irreducible projective characters corresponding to the $2$-cocycle $c_g$.
\cite{MR3715704} then establishes that all irreducible modules of $V^G$ are given by $V_{[g,\chi]}$ for some conjugacy class $g$ and some projective irreducible representation $\chi$ of $C_g$. In summary, the irreducible modules are labeled by $[g,\chi]$ and can be obtained from the twisted modules using the projector
\be
\pi_{\chi} = \frac{1}{|C_g|}\sum_{h\in C_g} \overline\chi(h) \phi_g(h)
\ee
In particular, this operator is independent of the choice of actions on the twisted modules.

\subsection{MTCs and holomorphic extensions}
The theory of holomorphic extensions of $V^G$ was established in \cite{Evans:2018qgz}.
If $V$ is tame and holomorphic, then $V^G$ is again tame, but no longer holomorphic. Its irreducible modules form a modular tensor category governed by a three cocycle $\omega \in H^3(G;\C^*)$, isomorphic to a twisted Drinfeld double of the group $G$, $D^\omega(G)$-mod. This was originally proposed in
 \cite{Roche:1990hs}, following up on work on the operator algebra of general orbifolds \cite{Dijkgraaf:1989hb}.
It has been established for $\omega$ trivial \cite{MR1923177}, and is believed to hold in general. For all orbifold VOAs $V^{orb(G)}$ we construct, $\omega$ will always be trivial, for reasons outlined below.

The cocycle $\omega$ then determines both the projective representations of the twisted modules, and the modular transformation properties of their characters:
Denote $g^h = h^{-1}gh$, ${}^hg = hgh^{-1}$. From $\omega$ we obtain a family of two cocycles $c_g, g \in G$ via
\be\label{descend}
c_g(h_1,h_2) = \omega(g,h_1,h_2)\omega(h_1,h_2,g^{h_1h_2})
\omega(h_1,g^{h_1},h_2)^*
\ee
In particular we will assume in the following that the  projective representation $\phi_g$ in the $g$-twisted module is chosen such that its 2-cocycle is given by (\ref{descend}). These in turn determine the modular transformation properties of the characters of $[a,\chi]$ through the $S$ and $T$ matrices
\bea\label{SGannon}
S_{[a_1,\chi_1],[a_2,\chi_2]}&=&\frac{1}{|G|}\sum_{g_i\in cl(a_i), g_1g_2=g_2g_1}
\chi_1(h_1)^*\chi_2(h_2)^* \frac{c_{g_1}(k_1^{-1},h_1)c_{g_2}(k_2^{-1},h_2)}{c_{g_1}(g_2,k_1^{-1})c_{g_2}(g_1,k_2^{-1})}\\ \label{TGannon}
T_{[a_1,\chi_1],[a_2,\chi_2]}&=& e^{-2\pi ic/24}\delta_{[a_1,\chi_1],[a_2,\chi_2]}\chi_1(a_1)/\chi_1(1)
\eea
Here we define $k_i$ through $g_i = a_i^{k_i}$, and $h_1 :={}^{k_i}g_2, h_2 :={}^{k_2}g_1$.

As mentioned above, $V^G$ is tame, but not holomorphic. Our goal is to extend $V^G$ to a holomorphic VOA $V^{orb(G)}$ by adjoining a suitable set of irreducible $V^G$-modules $[a,\chi]$ to $V^G$. This is not always possible, as there can be obstructions. In physics such cases are called `anomalous'.
The theory of such extensions is described in \cite{Evans:2018qgz}.
Holomorphic extensions are given by Corollary 2 of \cite{Evans:2018qgz}:
They occur if $\omega$ is trivial. In such a situation we can also choose `discrete torsion' $\psi\in Z^2(G,U(1))$. The extension is given by
\be
V^{orb(G),\psi}= \bigoplus_{k} W_{[k,\beta_k^{\psi}	]}
\ee
where $\beta_g^\psi(h) = \psi(g,h^g)\psi(h,g)^*$. (Note that \cite{Evans:2018qgz} allowed for a choice of subgroups $K < G$. We always choose $K=G$, since we will consider all possible subgroups anyway.) If $H^2(G,U(1))$ is non-trivial, we can thus recover several possibly non-isomorphic extensions. In all cases we consider however $H^2(G,U(1))$ is either trivial or does not lead to different VOAs, so that we will drop $\psi$ in our future notation and simply write $V^{orb(G)}$.

To have a holomorphic extension of $V^G$, it is thus necessary that $\omega$ is in the trivial class. We can then choose the representative $\omega=1$, so that $c_g(\cdot,\cdot)=1$, which means that the projective characters $\chi$ are actually linear. The extension with discrete torsion $\psi=1$ then simply consists of the $C_g$-invariant part of the $g$-twisted modules $W_g$,
\be
V^{orb(G),1} = \bigoplus_k W_k^{C_k}\ .
\ee
Moreover the transformation matrices are given by
\bea
S_{[a_1,\chi_1],[a_2,\chi_2]}&=&\frac{1}{|G|}\sum_{g_i\in cl(a_i), g_1g_2=g_2g_1}
\chi_1(h_1)^*\chi_2(h_2)^* \\
T_{[a_1,\chi_1],[a_2,\chi_2]}&=& e^{-2\pi ic/24}\delta_{[a_1,\chi_1],[a_2,\chi_2]}\chi_1(a_1)/\chi_1(1)
\eea
In principle we can determine if $[\omega]=1$ by constructing all twisted modules, from which we can extract the $c_g$ and the transformation matrices $S$ and $T$. We can then check if these data are compatible with $[\omega]=1$. For the cases we consider however, it turns out that $\omega$ is determined by its restriction to various cyclic subgroups. To test for anomalies of these cyclic orbifolds, we instead use the approach of \cite{vanEkeren2017}, which we describe now.

\subsection{Cyclic Orbifolds}
For cyclic groups $\Z_N$ this was worked out explicitly in \cite{vanEkeren2017}.
Let $a$ be a generator of $\Z_N$.
\begin{thm}[\cite{deWildPropitius:1995cf}]
	 The cohomology of $\Z_N$ is given by
	\begin{equation}
	H^3(\Z_N,U(1)) \cong \Z_{N}
	\end{equation}
 and is	generated by the $3$-cocycle
	\begin{equation}
	\omega(a^j,a^k,a^l) = \exp\big(\frac{2 \pi i}{N^2}j[k+l - \langle k+l\rangle_N]\big)\ .
	\end{equation}
	where $\langle \cdot \rangle_N$ denotes reduction modulo $N$.
\end{thm}
A general 3-cocycle $\omega_r$, $r=0,\ldots N-1$,
\begin{equation}
\omega_r(a^j,a^k,a^l) = \exp\big(\frac{2 \pi i r}{N^2}j[k+l - \langle k+l\rangle_N]\big)\ ,
\end{equation}
descends to the 2-cocycles
\be\label{cocyclecyclic}
c_{a^j}(a^k,a^l)=
\exp\big(\frac{2 \pi i r}{N^2}j[k+l - \langle k+l\rangle_N]\big)\,
\ee
For $[a^j,\chi_l]$ the linear characters are $\hat\chi_l(a^k)= e^{2\pi i lk/N}$ for $l\in\Z_N$, and the projective characters with cocycle (\ref{cocyclecyclic}) are $\chi_l(a^k) = e^{2\pi i (rjk/N+lk)/N}$, so that the modular data is given by
\be
T_{[a^j,\chi_l],[a^j,\chi_l]} = e^{-2\pi i c/24}\exp(2\pi i (rj^2+Njl)/N^2)\ .
\ee
and
\be
S_{[g_1,\chi_1],[g_2,\chi_2]}=\frac{1}{N}\chi_1(g_2)^*\chi_2(g_1)^*
\ee
On the other hand, \cite{vanEkeren2017} computed the  $T(g,h,\tau)$ for any $g,h\in \Z_n$ from the untwisted characters $T(e,g;\tau)$ by using the transformations
\begin{equation}\label{ScheitT}
T(g,h,T\cdot\tau) =
\sigma(g,h,T)T(g,gh,\tau)
\end{equation}
and
\begin{equation}\label{ScheitS}
T(g,h,S\cdot\tau) = %\exp((2\pi i)(-2rab/n^2))
\sigma(g,h,S)T(h,g^{-1},\tau)\ .
\end{equation}
The type $r \in \Z_N$ of the orbifold is defined as
\be
r = N^2 \rho_a \mod N\ ,
\ee
where $\rho_a$ is the conformal weight of the $a$-twisted module $W_a$. The
complex numbers $\sigma(g,h)$, \ie the modular data, are then fixed by $r$ of the orbifold.
In particular, using the explicit expressions for the $\sigma$, \cite{vanEkeren2017} showed that $\omega$ is trivial if and only if $r=0$, so that holomorphic extensions of $V^{\Z_N}$ can be constructed. This is the type 0 condition, and is equivalent to what is called `level matching' in physics.

\section{Effectively cyclic orbifolds}\label{s:CC}
\subsection{Effectively cyclic groups}
For cyclic examples we saw that it was not necessary to explicitly construct the twisted modules, since they could be obtained from the untwisted modules by modular transformation. This trick actually works for a much larger class of examples:
\begin{defn}\label{def:effcyc}
	Let $G$ be a finite group such that for all pairs $g \in G$ and $h\in C_g$  we can find a cyclic subgroup $C< G$ such that $g,h\in C$. We then say that $G$ is an \emph{effectively cyclic} orbifold group.
\end{defn}
The motivation for definition~\ref{def:effcyc} is that for such orbifolds, any twisted twining character $T(g,h;\tau)$ can be obtained from the orbifold characters of the cyclic group $C$. It is thus again unnecessary to explicitly construct the twisted modules.

One class of examples are  groups for which all $C_g$ except for the centralizer of the identity are cyclic. Such groups are  a special case of so-called CA groups, for which all centralizers are Abelian \cite{MR815926}.
More generally, it turns out that effectively cyclic is equivalent to another notion in group theory, namely periodic cohomology:
\begin{thm}
	Let $G$ be a finite group. $G$ is effectively cyclic iff $G$ has periodic cohomology.
\end{thm}
\begin{proof}
	Periodic cohomology is equivalent to every Abelian subgroup being cyclic, see \eg \cite{MR2012779}. Let $G$ have periodic cohomology. Then for any two commuting elements $g$ and $h$ the subgroup $\langle g,h\rangle$ is Abelian and hence by hypothesis cyclic. $G$ is thus effectively cyclic.

	Conversely, let $G$ be effectively cyclic and let $A$ be any abelian subgroup. Then by the fundamental theorem of finitely generated abelian groups $A$ is generated by elements $h_1, \ldots, h_n$ all of whose orders are prime powers. Now let $h_1$ and $h_2$ have orders $p^a$ and $p^b$, respectively. By hypothesis, there exists a cyclic group $C$ containing $h_1$ and $h_2$. If the order of $C$ is not a prime-power, $h_1$ and $h_2$ are elements of its $p$-primary part. Hence we may assume that $C \cong \Z_{p^c}$.  Since the subgroups of $\Z_{p^c}$ are well-ordered with respect to inclusions either $\langle h_1 \rangle$ is a subgroup of $\langle h_2 \rangle$ or vice versa. Hence one of the generators is redundant.
	By induction, for every prime $p$ there is a single generator whose order is a power of $p$ and hence $A$ is cyclic.
\end{proof}

Note that groups with periodic cohomology have trivial $H^2(G,U(1))$. This agrees with the fact that effectively cyclic orbifolds cannot have discrete torsion, since their characters are completely fixed by the untwisted modules.
Note that the converse however is not true: The group $(\Z_7 \times \Z_7) \rtimes \Z_3$, where $\Z_3$ acts on both $\Z_7$-factors by squaring, has trivial $H^2(G,U(1))$, but is not effectively cyclic.

\subsection{Semidirect products groups $\Z_q \rtimes_\phi \Z_p$}\label{ss:effcyc}
Let us now discuss a large class of such effectively cyclic orbifold groups which are of the form  $\Z_q\rtimes_\phi \Z_p$.

\begin{defn}\label{defGqp}
	Let $\phi \in \Z$ satisfy
	\begin{equation}\label{eq:phi}
	\phi^p \equiv 1 \pmod{q}.
	\end{equation}
	and let $(\phi-1)$ be coprime to $q$.
	Then the following relations define a group $\Zqp$:
	\begin{align}
	a^q & = 1 \label{rel:order_q_init} \\
	A^p & = 1 \label{rel:order_p_init} \\
	AaA^{-1} & = a^\phi \label{rel:comm_init}
	\end{align}
\end{defn}

If $\phi=1$, then of course this simply gives the direct product.
These groups produce many effectively cyclic examples:

\begin{thm}\label{thm:effcyc}
	The groups $\Zqp$ with $p,q$ coprime are effectively cyclic.
\end{thm}
\begin{proof}
	If $\phi\neq 1$ then this is exactly class 1 in the table on page 176 in \cite{MR2742530}. If $\phi=1$, then $\Z_q\rtimes_\phi \Z_p \equiv \Z_{pq}$.
\end{proof}

Now we want to describe the 3-cohomology classes for $\Z_q \rtimes_{\phi} \Z_p$ where $q$ and $p$ are coprime

\begin{lem}{\cite{wall_1961}}
The third cohomology group for the group $\Z_q \rtimes_{\phi} \Z_p$ such that $q$ and $p$ are coprime is given by
\begin{equation}\label{cohompeven}
	H^3(\Z_q \rtimes_{\phi} \Z_p,U(1)) \cong \Z_p \times \Z_r,
\end{equation}
where $r = (\phi^2-1,q)$.
\end{lem}

\begin{prop}
$H^3(\Z_q \rtimes_{\phi} \Z_p,U(1))$ is generated by the $3$-cocycles
\begin{equation}
\omega_p((a^i,A^I)(a^j,A^J)(a^k,A^K)) = \exp\big(\frac{2 \pi i}{p^2}I[J+K - \langle J+K\rangle_p]\big)
\end{equation}
and
\begin{equation}
\omega_r((a^i,A^I)(a^j,A^J)(a^k,A^K)) = \exp\big(\frac{2 \pi i}{r^2}\phi^{J+K}i[\phi^{K}\langle j\rangle_r+\langle k\rangle_r - \langle \phi^{K}j+k\rangle_r]\big)
\end{equation}
\end{prop}
\begin{proof}
$\omega_p$ is a $3$-cocycle and together both cocycles have the correct order to generate the group. It remains to be shown that $\omega_r$ is indeed a $3$-cocycle.
We calculate
\begin{align*}
\delta\omega_r&((a^i,A^I)(a^j,A^J)(a^k,A^K)(a^l,A^L)) \\
& = \exp\bigg(\frac{2\pi i}{r^2}\bigg[\phi^{J+K}i[\phi^K\langle j\rangle_r+\langle k\rangle_r- \langle\phi^Kj+k\rangle_r]
+ \phi^{J+K+L}i[\phi^L\langle\phi^Kj+k\rangle_r+\langle l\rangle_r - \langle\phi^L(\phi^Kj+k)+l\rangle_r] \\
&\quad\quad\quad\quad\quad+ \phi^{K+L}j[\phi^L\langle k\rangle_r+\langle l \rangle_r-\langle\phi^Lk+l\rangle_r]
- \phi^{K+L}(\phi^Ji+j)[\phi^L\langle k \rangle_r+\langle l\rangle_r-\langle\phi^Lk+l\rangle_r] \\
&\quad\quad\quad\quad\quad- \phi^{J+K+L}i[\phi^{K+L}\langle j\rangle_r+\langle\phi^Lk+l\rangle_r - \langle\phi^{K+L}j+ \phi^Lk+l\rangle_r] \bigg]\bigg)\\
& = \exp\bigg(\frac{2\pi i}{r^2}(1-\phi^{2L})\phi^{J+K}i(\phi^K\langle j\rangle_r+\langle k\rangle - \langle\phi^Kj+k\rangle_r)\bigg) \\
& = 1,
\end{align*}
where we make use of the facts that $r$ divides both $1-\phi^{2L}$ and $\phi^K\langle j\rangle_r+\langle k\rangle_r - \langle\phi^Kj+k\rangle_r$.
\end{proof}

For the groups listed in on section~\ref{s:lattice}, we find there is a simple criterion for determining whether the cohomological twist is trivial.

\begin{cor}
	If $q$ is squarefree, $H^3(G,U(1))$ is cyclic, and $\omega \in H^3(G,U(1))$ is trivial if and only if the type 0 condition is satisfied for the cyclic subgroups $\Z_p$ and $\Z_r$.
\end{cor}
\begin{proof}
The restrictions of $\omega_p$ and $\omega_r$ to the cyclic subgroups $\Z_p$ and $\Z_r \subset \Z_q$ respectively are exactly the generators of the corresponding third cohomology groups. Hence their contribution to the cohomological twist can be determined from the respective types as in Equation \ref{cocyclecyclic}.
\end{proof}

There is one other case which we will use later on: namely, if $p=q$, and $\phi=1$. In that case we recover the Abelian group $\Z_p\times\Z_p$, which is not effectively cyclic. The generators of its cohomology are given by the following theorem:

\begin{thm}[\cite{deWildPropitius:1995cf}]\label{H3ZpZp}
	For abelian groups of the form $\Z_p \times \Z_p$
	the third cohomology group is given by
	\begin{equation}
	H^3(\Z_p \times \Z_p,U(1)) \cong \Z_p^3
	\end{equation}
	and is generated by
	\begin{equation}
	\omega^{(1)}((a^i,A^I)(a^j,A^J)(a^k,A^K)) = \exp\big(\frac{2 \pi i}{p^2}i[j+k - \langle j+k\rangle_p]\big)\ ,
	\end{equation}
	\begin{equation}
	\omega^{(2)}((a^i,A^I)(a^j,A^J)(a^k,A^K)) = \exp\big(\frac{2 \pi i}{p^2}I[J+K - \langle J+K\rangle_p]\big).
	\end{equation}
	and
	\begin{equation}
	\omega^{(12)}((a^i,A^I)(a^j,A^J)(a^k,A^K)) = \exp\big(\frac{2 \pi i}{p^2}i[J+K - \langle J+K\rangle_p]\big),
	\end{equation}
\end{thm}
Note that this implies that $\omega=1$ if and only if the cyclic groups $\langle a \rangle, \langle A \rangle$ and $\langle aA \rangle$ are all of type 0.

\subsection{Characters of effectively cyclic orbifolds}
Let us finally give a closed form expression for the character of $V^{orb(G)}$ for such groups in terms of cyclic orbifold characters. To do this, we want to compute the conjugacy classes and centralizers of the $\Z_q \rtimes_{\phi} \Z_p$ groups we introduced above.
To do this, we introduce a slightly different notation:
Let $p$ be coprime to $q$,
$1\neq \phi \in \Z$ as above, and $r$ the smallest divisor of $p$ such that $\phi^r \equiv 1 \pmod{q}$.
\\
\begin{lem}
	The following relations define the same group $\Z_q \rtimes_{\phi} \Z_p$ as above:
	\begin{align}
	a^q & = 1\\
	A^r & = B \\
	B^{\frac{p}{r}} & = 1 \\
	AaA^{-1} & = a^\phi \\
	BaB^{-1} & = a
	\end{align}
	Then, in particular, $\phi^r \equiv 1 \pmod{q}$.
\end{lem}
In what follows we take $q$ to be prime. This will cover most of the cases we are interested in, and it turns out that the expression for the character we derive still holds in the cases when $q$ is not prime.

A general element $g \in \Z_q \rtimes_{\phi} \Z_p$ can be expressed as $g = a^uA^v B^w$ with $u = 0,\ldots,q-1$, $v = 0,\ldots,r-1$ and $w = 0,\ldots,\frac p r - 1$.
The center of $\Zqp$ is given by $\Z_{\frac{p}{r}} = \langle B \rangle$.

\begin{thm}
	The conjugacy classes and centralisers of $\Z_q \rtimes_{\phi} \Z_p$ are given by
	\begin{enumerate}
		\item A class $[A^i]$ of length $q$ for every $i\in \Z_p$ with $r\nmid i$  with centraliser $C_{A^i} = \Z_p$.
		\item A class $[a^iB^j]$ of length $r$ for every pair $(a^i,B^j)$ of a coset representative $a^i \in \frac{\Z_q^*}{\langle \phi \rangle_{\text{mult}}}$ and an (possibly identity) element $B^j \in \Z_{\frac{p}{r}}$ with centraliser $C_{\Z_q \rtimes_{\phi} \Z_p}(a^iB^j) = \Z_{\frac{pq}{r}}$
		\item A central class $[B^i]$ of length $1$ for every element $B^i \in \Z_{\frac{p}{r}}$
	\end{enumerate}
\end{thm}

\begin{proof}
	For a general element conjugation is given by
	\begin{equation}
	a^xA^y a^uA^v B^w A^{-y}a^{-x} = a^{u\phi^y-x(\phi^v-1)}A^vB^w.
	\end{equation}
	First take $v \neq 0$. Then $\phi^v - 1$ is invertible modulo $q$  and
	the element $a^uA^vB^w$ is conjugate to $A^vB^w$ for any $u$.
	The centralisers of these representatives are exactly the group $\Z_p = \langle A \rangle$, so that the length of $[A^i]$ is $q$.

	Next take $v=0$, $u \neq 0$. Then
	\begin{equation}
	a^xA^ya^uB^wA^{-y}a^{-x} = a^{u\phi^y}B^w.
	\end{equation}
	shows that there are $\frac{q-1}{r}\times \frac{p}{r}$ conjugacy classes  representatives
	$a^uB^w$, $u \in \frac{\Z_q^*}{\langle \phi \rangle}$.
	The centralisers of those conjugacy classes are generated by $a$ and $B$ and are therefore the cyclic group $\Z_{\frac{pq}{r}} = \langle aB \rangle$.

	Finally, $u=v=0$ gives the $\frac{p}{r}$ central classes $B^w$ whose centralizer is of course the entire group $\Zqp$.
\end{proof}

\begin{thm}\label{effcycformula}
	Assume that $\omega=1$ for the orbifold $\Zqp$.
	The character of the orbifold $V^{\text{orb}(\Zqp)}$ is given by
	\begin{align}\label{orbifoldcharacter}
	\chi^{\text{orb}(\Zqp)}(\tau) = \chi^{\text{orb}(\Z_p)}(\tau) - \frac{1}{r}(\chi^{\text{orb}(\Z_\frac{p}{r})}(\tau) - \chi^{\text{orb}(\Z_\frac{pq}{r})}(\tau)).
	\end{align}
\end{thm}
\begin{proof}
	Let $W_z^G(\tau)$ be the $G$-invariant of the irreducible $z$-twisted module.
	By \cite{Evans:2018qgz} the orbifold character is given by
	\begin{align}
	\chi^{\text{orb}(\Z_q \rtimes_{\phi} \Z_p)}(\tau) & =  \sum_{g \in \text{cl}(\Zqp))} W_g^{C_g}(\tau) \\
	& =  \sum_{B^j \in \Z_\frac{p}{r}} W_{B^j}^{\Zqp}(\tau) + \sum_{A^j,r \nmid j} W_{A^j}^{\Z_p}(\tau)
	+ \sum_{a^iB^j:a^i \in \frac{\Z_q^*}{\langle \phi \rangle}, B^j \in \Z_\frac{p}{r}} W_{a^iB^j}^{\Z_{\frac{pq}{r}}}(\tau) \\
	& =  \sum_{B^j \in \Z_\frac{p}{r}} W_{B^j}^{\Zqp}(\tau) + \sum_{A^j,r \nmid j} W_{A^j}^{\Z_p}(\tau)
	+ \frac{1}{r}\sum_{a^iB^j:a^i \in \Z_q^*, B^j \in \Z_\frac{p}{r}} W_{a^iB^j}^{\Z_{\frac{pq}{r}}}(\tau)
	\end{align}
	where in the last line we used the fact that $|\Z^*_q/\langle \phi \rangle| = |\Z^*_q|/r$. To rewrite the character invariant under the entire group $\Zqp$, we use the fact that it has $q$ maximal subgroups of order $p$ that are all conjugate to $\Z_p = \langle A \rangle$, and one maximal subgroup $\Z_{\frac{pq}{r}} = \langle aB \rangle$ of order $\frac{pq}{r}$.  $\Zqp$ is the union of all these maximal subgroups, and their pairwise intersections all equal to the center $\langle B \rangle$. Using the fact that the twining characters $T(g,h;\tau)$ only depend on the conjugacy class of $h$, we can thus decompose the first character as
	\begin{align}
	W_{B^j}^{\Zqp}(\tau) & = \frac{1}{pq}\big(q \sum_{d = 0}^{p-1} T(B^j,A^d,\tau) + \sum_{b = 0}^{\frac{pq}{r}-1}T(B^j,(aB)^b,\tau)\big) - q\sum_{c = 0}^{\frac{p}{r}-1} T(B^j,B^c,\tau) \\
	& = \frac{1}{pq}\big(pq W_{B^j}^{\Z_p}(\tau) + \frac{pq}{r}W_{B^j}^{\Z_{\frac{pq}{r}}}(\tau) - \frac{pq}{r}W_{B^j}^{\Z_{\frac{p}{r}}}(\tau)\big) \\
	& =  W_{B^j}^{\Z_p}(\tau) - \frac{1}{r}\big(W_{B^j}^{\Z_{\frac{p}{r}}}(\tau) - W_{B^j}^{\Z_{\frac{pq}{r}}}\big).
	\end{align}
	Hence
	\begin{align}
	\chi^{\text{orb}(\Zqp)}(\tau) & =  \sum_{B^j \in \Z_\frac{p}{r}} W_{B^j}^{\Z_p}(\tau) - \frac{1}{r}\big(W_{B^j}^{\Z_{\frac{p}{r}}}(\tau) - W_{B^j}^{\Z_{\frac{pq}{r}}}\big) \\
	& \quad  + \sum_{A^j,r \nmid j} W_{A^j}^{\Z_p}(\tau)
	+ \frac{1}{r}\sum_{a^iB^j:a^i \in \Z_q^*, B^j \in \Z_\frac{p}{r}} W_{a^iB^j}^{\Z_{\frac{pq}{r}}}(\tau)  \\
	& = \sum_{A^j \in \Z_p} W_{A^j}^{\Z_p}(\tau) + \frac{1}{r}\sum_{a^iB^j \in \Z_\frac{pq}{r}} W_{a^iB^j}^{\Z_{\frac{pq}{r}}}(\tau) - \frac{1}{r}\sum_{B^j \in \Z_\frac{p}{r}}W_{B^j}^{\Z_{\frac{p}{r}}}(\tau) \\
	& = \chi^{\text{orb}(\Z_p)}(\tau) - \frac{1}{r}(\chi^{\text{orb}(\Z_\frac{p}{r})}(\tau) - \chi^{\text{orb}(\Z_\frac{pq}{r})}(\tau)).
	\end{align}
\end{proof}
There are some cases in section~\ref{s:lattice} for which $q$ is not prime. They turn out to be of the form $\Z_q\rtimes \Z_4$ and $\phi=-1$, and it turns out that (\ref{orbifoldcharacter}) is still correct for them.

\section{Lifting theory}\label{s:lifting}
\subsection{Lifting map}

Let us now discuss orbifolds of lattice VOAs. The idea is to obtain automorphisms of the lattice VOA $V_L$ by lifting automorphisms of the underlying lattice $L$. We will now denote VOA automorphisms with hatted letters, and the underlying lattice automorphism with unhatted letters.

From \cite{MR1745258} we know that for lattice VOAs, $\Aut(\hat L)$ and $\Aut(L)$ are related by
\be\label{ses}
1 \to  \Hom(L,\C^*)\to \Aut(\hat L)\to \Aut(L)\to 1
\ee
To proceed, we pick a subgroup $G<\Aut(L)$ to get the short exact sequence
\be\label{Gses}
1 \to  \Hom(L,\C^*)\to \hat G \to G\to 1\ ,
\ee
with $\hat G < \Aut(\hat L)$. To lift $G$, we construct a lift
\be
\iota : G \to \hat G \qquad g \mapsto \gl := (1,g)\ ,
\ee
where $\gl\in Aut(\hat L)$ acts on $\hat L$ as
\be
\gl e_\alpha = \eta_g(\alpha)^{-1}e_{\alpha g}\ .
\ee
Here $\eta_g(\cdot)$ is a function from $L$ to $\C^*$ satisfying
\be\label{etacondition}
\frac{\eta_{g}(\alpha)\eta_{g}(\beta)}{\eta_{g}(\alpha+\beta)} = \frac{\epsilon(\alpha,\beta)}{\epsilon(g\alpha,g\beta)}\ .
\ee
This condition ensures that $\hat g$ is indeed an automorphism of the VOA.
Note that this implies that in general $\eta_g$ is not in $\Hom(L,\BC^*)$.
We derive an explicit expression for a lifting map $\iota$, that is for a $\eta_g(\cdot)$, in theorem~\ref{thm:etaexpression} in appendix~\ref{app:lift}.
This $\eta_g$ is of course not unique: for any $\xi_g \in \Hom(L,\BC^*)$, $\xi_g\eta_g$ again satisfies (\ref{etacondition}). Let us therefore write $\hat g = (\xi_g,g)$ for general elements in $Aut(\hat L)$
with action on $\hat L$ given by
\be
\hat g e_\alpha =(\xi_g,g) e_\alpha= \xi_g(\alpha)^{-1}\eta_g(\alpha)^{-1} e_{\alpha g}\ .
\ee
Their group multiplication is given by
\be
\hat g \cdot \hat h =
(\xi_g,g)\cdot (\xi_h,h)= (\xi_g\xi_h s(g,h),gh)
\ee
where
\be
s(g,h)=\frac{\eta_{h}((g h)^{-1}\cdot)\eta_{g}(g^{-1}\cdot)}{\eta_{g h}((g h)^{-1}\cdot)}\ ,
\ee
which, as is shown in appendix~\ref{app:liftmult}, is indeed in $\Hom(L,\BC^*)$.
In practice we describe $\xi_g\in \Hom(L,\BC^*)$ by the choice of a vector $v_g$ such that $\xi_g(\alpha)= e^{2\pi i\langle\alpha,v_g\rangle}$.

We call a lift such that
\be
\eta_g(\alpha)=1\ , \qquad \alpha \in L^g
\ee
a \emph{standard lift}. In what follows we will always work with such standard lifts. Given an arbitrary lift $\eta_g$, we can turn it into a standard lift $\tilde \eta_g$ by
the algorithm given in appendix~\ref{app:standardlift}, which constructs a homomorphism $\xi_g$ such that $\xi_g \eta_g$ is a standard lift.
Often the mere existence of a standard lift is enough for the construction of twisted module characters; in particular it turns out that for the construction described in section~\ref{s:CC}, we only need the action on the fixed point lattice.

In the end of the day we are not only interested in the lifted elements $\iota(g)$ by themselves, but in the group $\langle \iota(G) \rangle$ that they generate.
We will be particularly interested in the case where $\iota$ splits (\ref{Gses}), so that the lifted group $\iota(G)$ is isomorphic to $G$.
If $\iota$ is not splitting, we will have to work with the lifted group $\langle \iota(G) \rangle < \hat G$ instead, which is bigger than $G$. A typical case when this can occur is if $G$ is cyclic, generated by an element $g$ of even order. To satisfy (\ref{etacondition}), it is then sometimes necessary to define $\eta_g(\cdot)$ such that the order of $\gl$ is doubled, $ord(\gl) = 2ord(g)$ \cite{vanEkeren2017}.

\subsection{Lifting effectively cyclic groups}\label{ss:effcyclift}
Finding a lift that splits is particularly important if $G$ is effectively cyclic and we want  $\iota(G)$ to remain effectively cyclic.
In view of section~\ref{ss:effcyc}, let us discuss the lifting theory of groups of the form
\be
G = \Z_q\rtimes_{\phi}\Z_p\ .
\ee
In appendix~\ref{app:liftZqZp} we discuss how to construct a splitting lift such that
\be
\iota(\Z_q\rtimes_{\phi}\Z_p) \cong \Z_q\rtimes_{\phi}\Z_p\ .
\ee
The outcome is that in principle there can be an obstruction to such a splitting lift. We will work with lifts that leave the order of the generators invariant, such that $\gl$ and $\hl$ still satisfy (\ref{rel:order_p_init}) and (\ref{rel:order_q_init}). To ensure (\ref{rel:comm_init}), we use the following lemma:

\begin{lem}
	The relation $\gl\hl\gl^{-1}  = \hl^{\phi}$ is equivalent to $\eta_{g,h}$ satisfying
	\begin{equation}\label{eq:lift_main}
	\eta_g(\alpha g^{-1}h)\overline{\eta}_g(\alpha g^{-1})\eta_h(\alpha g^{-1})\prod_{i=0}^{\phi-1}\overline{\eta}_h(\alpha h^i) = 1.
	\end{equation}
\end{lem}
To obtain such $\eta_{g,h}$, we construct homomorphisms $\xi_{g,h}$ such that the new lifts $\tilde \eta_g = \eta_g\xi_g$ and $\tilde \eta_h = \eta_h\xi_h$ do satisfy Equation \ref{eq:lift_main}. In order for the new $\tilde \eta_g$ and $\tilde \eta_h$ to still be standard lifts we demand that $\xi_g(L_g) = \xi_h(L_h) = 1$.
We find that even though there could be an obstruction to finding such $\xi_{g,h}$, for all examples we consider, this obstruction vanishes, so that we can find a splitting lift.

For an effectively cyclic orbifold as described in Theorem~\ref{thm:effcyc},  note that the mere existence of this lift is enough, since we only need to know its action on the untwisted sector.
For groups
that do not fall in this category, such as for instance $\Z_3\times\Z_3$, we will need the explicit expression for the lift, which is provided in lemma~\ref{ZqZplift}.

\section{Constructing twisted modules}\label{s:twistedmodules}

The twisted modules $W_{\gl}$ for lattice VOAs were constructed in  \cite{MR820716,MR2172171} as induced representations of a smaller representation $\Omega_0$ of a finite group, the so called defect representation. We will briefly review this construction, using the notation of \cite{MR2172171}, and then give an explicit construction of $\Omega_0$ using technology from \cite{MR2742735}.

\subsection{The $g$-twisted heisenberg current algebra $\hat{\mathfrak{h}}_g$}

Let $\mathfrak{h} = L \otimes_{\Z} \C$  and let $g$ act linearly on $\mathfrak{h}$. Let $\mathfrak{h}_j$ be the eigenspace of $g$ to the eigenvalue $\xi_n^{-j}$.
\begin{defn}\label{def:heisenberg}
	The $g$-twisted heisenberg current algebra $\hat{\mathfrak{h}}_g$ consists of states
	\begin{equation}
		\hat{\mathfrak{h}}_g = \text{span}\{v_m:v \in \mathfrak{h},m \in \frac{j}{n} + \Z \text{ if } gv = \xi_n^{-j}v\}.
	\end{equation}
	As an operator $v_m$ has weight $-m$.
\end{defn}

\subsection{Twisted lattice operators $U_\alpha$}
Let $L$ be an even, unimodular lattice, $g$ a lattice automorphism of order $n$ . Also let $V_L$ be the lattice VOA corresponding to $L$ and $\gl$ be a standard lift of $g$.
We denote the unique irreducible $\hat g$-twisted $V_L$-module by $W_{\gl}$.
To construct the twisted lattice vertex operators,
we define the group $G_\gl$, whose elements are of the form $c e^v U_\alpha \in \C^\times \times \exp \mathfrak{h}_0 \times L$, where $\mathfrak{h}_0 = \pi_g \mathfrak{h}$ the orthogonal projection on $g$-invariant subspace, together with the relations
\be\label{meq:U_mult}
	U_{\alpha}U_{\beta} = \epsilon(\alpha,\beta)B(\alpha,\beta)^{-1} U_{\alpha + \beta},
	\ee
	where $\epsilon(\alpha,\beta)$ is defined in Definition \ref{def:eps} and
	\be\label{meq:B_def}
	B(\alpha,\beta) = n^{-\langle \alpha | \beta \rangle} \prod_{k=1}^{n-1}(1-\xi_n^k)^{\langle \alpha g^k | \beta\rangle},
	\ee
	\begin{equation}
		e^{v}e^{w} = e^{v+w},
	\end{equation}
	\begin{equation}
		e^{v}U_{\alpha}e^{-v} = \exp(\langle v|\alpha \rangle)U_{\alpha}
	\end{equation}
	and  the twist compatibility condition
	\begin{equation}\label{meq:U_tw}
	U_{\alpha g} = \eta_g(\alpha)U_{\alpha}\exp\big(2 \pi i(b_{\alpha} + \alpha_{(0)})\big),
	\end{equation}
	where
	\begin{equation}
	b_{\alpha} = \frac{|\pi_g(\alpha)|^2}{2}\in \C.
	\end{equation}
	Additionally, let the twisted lattice operators act on $\hat{\mathfrak{h}}_g$ by
	\begin{equation}
	[h_{(m)},U_{\alpha}] = \delta_{m,0}\langle \pi_g(h)|\alpha\rangle U_{\alpha}, \quad \text{for all } \alpha \in L, h \in \mathfrak{h}.
	\end{equation}

	Finally, let us define $L_g^\perp \subset L$ to be the sublattice of coinvariant lattice vectors, that is $L_g^\perp := L \cap (1-\pi_g)\R^d$, where $\pi_g$ is the orthogonal projector onto the $g$ invariant subspace. Furthermore, define the corresponding subgroup of $G_\gl$
	\begin{equation}
		G_\gl^\perp = \{cU_{\mu}: c \in \C, \mu \in L_g^\perp\}.
	\end{equation}

	Then in \cite{MR2172171} it is shown that as a $\hat{\mathfrak{h}}_g$-module $W_{\gl}$ is induced from its vacuum subspace
		\begin{equation}
		\Omega = \{w \in W_{\gl}| v_m w = 0, v_m \in \hat{\mathfrak{h}}_g, m \geq 0\}.
		\end{equation}

		$\Omega$, in turn, is a $G_\gl$-module that decomposes into $\exp \mathfrak{h}_0$ eigenspaces
		\be
		\Omega = \sum_{\alpha\in \pi_g(L)} \Omega_\alpha\ ,
		\ee
		where $\Omega_{\alpha} = \{w \in \Omega|e^h w = e^{\langle h|\alpha \rangle}\}$ such that
		\begin{equation}
		U_\beta \Omega_{\alpha} = \Omega_{\alpha+\pi_g\beta}.
		\end{equation}
		It follows that all the homogeneous subspaces $\Omega_{\alpha}$ are $G_\gl^\perp$-modules and $\Omega$ can be induced from $\Omega_0$ (or any $\Omega_{\alpha}$ for that matter).

		Note that from (\ref{meq:U_tw}) it follows that $G_\gl^\perp$ is a central extension for the finite abelian group
		\be
			\FA := L_g^\perp/L(1-g)\
		\ee
		whose isomorphism class is determined by the skew
		\begin{equation}
		C(\alpha, \beta) = U_{\alpha}U_{\beta}U_{\alpha}^{-1}U_{\beta}^{-1} = \prod_{k=0}^{n-1}(-\xi_n^k)^{-\langle \alpha g^k | \beta \rangle}.
		\end{equation}
		In particular, Equation (4.44) in \cite{MR2172171} shows that $C(\alpha,\beta)$ restricts to a non-degenerate, alternating, bimultiplicative form on $\FA$, endowing it with the structure of a symplectic module.

		$W_{\gl}$ is irreducible if and only if $\Omega_0$ is an irreducible projective representation of $N$.
		We now want to construct $\Omega_0$ (or any $\Omega_\alpha$ for that matter)
		as a projective representation $\Omega_0$ of $\FA$  of dimension $d(g) = \sqrt{|N|}$. We will call $\Omega_0$ the defect representation.
		We will see that $\Omega_0$ is the the unique projective representation of $N$ with commutator
		 $C(\cdot,\cdot)$.

\subsection{Defect Representation}

To construct the representation $(\rho,\Omega_0)$ of $N$ we want to decompose $N$ in the following way.
Using the fact that $C$ is a non-degenerate, bimultiplicative, alternating form from the results in appendix~\ref{app:Davydov}, we know that there
exists a subgroup $A \le \FA$ and an isomorphism $\gamma$ such that
\be
\gamma: \FA \to A + \check A \qquad [\mu] \mapsto (x_{[\mu]},\chi_{[\mu]})\ ,
\ee
where $\check A$ denotes the group of irreducible characters of $A$, such that the bilinear form $C$ is given by
\be
C(\gamma^{-1}(x,\chi),\gamma^{-1}(y,\psi)) = \psi(x)\chi(y)^{-1}.
\ee
From this decomposition we can construct the sought after projective representation $\Omega_0$:
\begin{thm}
	The unique (up to projective equivalence), irreducible representation $(\rho,\Omega_0)$ of $\FA$ such that the skew of its $2$-cocycle is given by $C$ is given by
	\be
	\rho(x,\chi)\eD_{\psi} = \psi(x)\eD_{\chi\psi},
	\ee
	where $\Omega_0$ is the vector space generated by $\{\eD_{\chi}:\chi \in \check A\}$.
\end{thm}
\begin{proof}
We find that $\rho$ satisfies the product relation
\begin{align}
\rho(x,\chi)\rho(y,\phi)e_{\psi} & =\phi(x)\psi(xy)e_{\chi\phi\psi} \\
& = \phi(x) \rho(xy,\chi\phi)e_{\psi}
\end{align}
so that $\rho$ does indeed define a projective representation for the group $N$ and the commutator is indeed given by $C$.

Then the characters are given by
\begin{equation}
	\mathrm{Tr}(\rho(x,\chi)) = \begin{cases} \sum_{\psi \in \check{A}} \psi(x), \quad & \text{if } \chi = 1 \\ 0 \quad & \text{otherwise}. \end{cases}
\end{equation}

Hence the number of irreducible representations in $\rho$ is
\begin{align}
	\|\rho\|^2 & = \frac{1}{|A|^2} \sum_{x \in A} \overline{\mathrm{Tr}(\rho(x,0))}\mathrm{Tr}(\rho(x,0)) \\
	& = \frac{1}{|A|^2} \sum_{x \in A}\sum_{\chi,\psi \in \check{A}} \overline{\chi(x)}\psi(x) \\
	& = \sum_{\chi \in \check{A}}\frac{1}{|A|^2} \sum_{x \in A} \overline{\chi(x)}\chi(x) \\
	& = \sum_{\chi \in \check{A}}\frac{1}{|A|} \\
	& = 1,
\end{align}
so that $\rho$ is indeed irreducible. Finally, the uniqueness of $\rho$ follows from the fact that $\mathrm{dim}(\rho)^2 = |N|$.
\end{proof}
This representation $\rho$ is now almost the sought after representation of $L_gF^\perp$. The only problem is that if we try to define the $U_\mu$ in this way, they may not satisfy the twist compatibility condition (\ref{meq:U_tw}). Due to uniqueness, we can solve this problem by a change of section, which gives a new representation $\tilde \rho$ from $\rho$.
That is, we  lift $\rho$ to a representation $\tilde \rho$ of the twisted lattice operators $U_{\mu}$ for $\mu \in L^{\perp}_g$,
\be
\tilde \rho : L^{\perp}_g \to GL(\Omega_0) \qquad \mu \mapsto U_\mu\ .
\ee
$\tilde \rho$ is related to $\rho$ by change of section, i.e. there is a function $\lambda: L^{\perp}_g \to \C^*$ such that
\be\label{meq:lambda_def}
U_{\mu} = \lambda(\mu)\rho(x_{[\mu]},\chi_{[\mu]}),
\ee
where $(x_{[\mu]},\chi_{[\mu]}) = \gamma([\mu])$.
We need to be careful in our choice of $\lambda$, since the $U_\mu$ need to satisfy both (\ref{meq:U_mult}) as well as (\ref{meq:U_tw}). These are equivalent to $\lambda$ satisfying  the two conditions
	\begin{equation}\label{meq:lambda_mult}
\lambda(\mu)\lambda(\nu)\chi_{\nu}(x_{\mu}) = \epsilon(\mu,\nu)B(\mu,\nu)^{-1}\lambda(\mu+\nu)
\end{equation}
and
	\begin{equation}\label{meq:lambda_tw2}
\lambda(\alpha(1-g)) =\eta(\alpha)^{-1}\epsilon(\alpha(1-g),\alpha g)
B(\alpha (1-g),\alpha g)^{-1}\exp(2\pi i b_{\alpha})
\end{equation}

In order to proceed to find such a $\lambda$ first note the following:
\begin{lem}
	If the operators $U_{\mu}, U_\nu$ satisfy the twist compatibility condition (\ref{meq:U_mult}), then so does $U_{\mu+\nu}$.
\end{lem}
Hence if we determine $\lambda(\mu_i)$ for a basis $\mu_i$ of $L_g^{\perp}$ (and hence $N$) such that all $U_{\mu_i}$ satisfy twist compatibility, we can repeatedly apply (\ref{meq:U_mult}) to obtain a full set of twist compatible lattice operators. In particular, we show in Appendix \ref{app:twist} that a general lattice is given by
\begin{equation}
U_{(\sum_i n_i \mu_i)} = \prod_{i<j} \big(\epsilon(\mu_i,\mu_j)B(\mu_i,\mu_j)\big)^{n_in_j}
\prod_i U_{\mu_i}^{n_i}\big(\epsilon(\mu_i,\mu_i)B(\mu_i,\mu_i)\big)^{\frac{n_i(n_i-1)}{2}}.
\end{equation}

Theorem~\ref{thm:lambdabasis} gives an expression for the $\lambda$ on a basis of $N$. From this we can construct $U_{\mu_i}$ for a basis $\mu_i$, and thus by (\ref{meq:B_def}) all the operators $U_\mu$.

\subsection{A basis for $W_\gl$}
Finally let us give a basis for $W_{\gl}$ starting from $(\tilde\rho, \Omega_0)$.
First we find a lattice $\Lambda$ such that
\be\label{Lambda}
L= \Lambda + L_g^\perp
\ee
such that we can decompose $\alpha=\gamma+\mu$. Note that we cannot necessarily choose $\Lambda$ to be $L\cap \pi_g \R^d$.
Let $\{e_{\chi}:\chi \in \check A\}$ be a basis for $\Omega_0$. We can reconstruct $W_{\gl}$ by

\begin{thm}
	A basis for $W_{\gl}$ is given by
	\begin{equation}\label{Wgbasis}
		\{ v^1_{-m_1}\ldots v^k_{-m_k} U_{\gamma} e_{\chi}: v^i_{-m_i} \in \hat{\mathfrak{h}}_g, m_i > 0, \gamma \in \Lambda, \chi \in \check A \}.
	\end{equation}
	The weight of such a state is given by
	\be
	m_1 + \ldots + m_k + |(\pi_g \gamma)|^2/2 + \rho_g,
	\ee
	recovering the familiar result for the character.
\end{thm}
A general twisted lattice operator $U_\alpha$ then acts on (\ref{Wgbasis}) in the obvious way.

\section{Twining Characters}\label{s:chars}

\subsection{Twining elements}
Let $V$ be a $C_2$-cofinite VOA, $\gl$ and $\hl$ automorphisms of $V$ such that $\gl\hl = \hl\gl$ and $(W_{\gl},Y_W(\cdot,x))$ an irreducible $\gl$-twisted $V$-module.
Then by the results of \cite{1997q.alg.....3016D, MR2023933,MR3715704}
	there exists a  linear map $\phi_\gl(\hl):W_{\gl} \to W_{\gl}$ such that
	\be\label{eq:phi3}
	\phi_\gl(\hl)Y_W(v,x)\phi_\gl(\hl)^{-1} = Y_W(\hl v,x)\ ,
	\ee
which is unique up to multiplication by a scalar. $\phi_\gl(\cdot)$ is a projective representation of $C_\gl$.
We need to explicitly work out what this representation is for the construction at hand.

Let us now insert a twining element $\hat h\in C_\gl$.
Acting with $\phi_\gl(\hl)$ on a vector in (\ref{Wgbasis}) we can use (\ref{eq:phi3}) to find
\be
\phi_\gl(\hl) U_\alpha \phi_\gl(\hl)^{-1}= \eta_h(\alpha) U_{\alpha h}
\ee
and
\be
\phi_\gl(\hl) v_{-n} \phi_\gl(\hl)^{-1} = (v\cdot h)_{-n}\ .
\ee
We therefore have
\begin{lem}
	Let $\hat g, \hat h \in \mathrm{Aut}(V_L)$ be two automorphisms such that $\hat g \hat h = \hat h \hat g$ lifted from commuting lattice automorphisms $g,h \in \mathrm{Aut}(L)$. Then $W_{\gl}$ carries an action $\phi_{\hat g}(\hat h)$ of $\hat h$ given by
	\begin{equation}
	\phi_{\hat g}(\hat h)v_{-n_1}\ldots v_{-n_k} U_{\alpha} e_{\chi}
	= (vh)_{-n_1}\ldots (vh)_{-n_k} \eta_h(\alpha)^{-1}U_{\alpha h} \phi_{\hat g}(\hat h)|_{\Omega_0}e_{\chi}
	\end{equation}
\end{lem}
We thus need to find the action of $\phi_\gl(\hl)$ on $\Omega_0$. Note that  $\phi_\gl(\hl)|_{\Omega_0}$ is a $d(\gl)\times d(\gl)$-matrix $O_\hl \in GL(\Omega_0)$. To compute $O_\hl$,
we note that $\tilde \rho(\cdot h)$ is again an irreducible representation of $L_g^\perp$ of dimension $d(g)$ with the same cocycle, and in fact also an irreducible representation of $N$. Since there is only one irreducible representation of dimension $d(g)$, Schur's lemma tells us that $\tilde \rho(\cdot h)$ and $\tilde \rho(\cdot)$ are related by a change of basis automorphism $O_{\hl}$ on $\Omega$, which is unique up to an overall factor in  $\C^*$. From this it follows immediately that  $\phi_\gl|_{\Omega_0} : \hl \mapsto O_{\hl}$ is indeed a projective representation of $C_\gl$ on $\Omega_0$.

To construct $O_\hl$, we construct $U_{\mu_i}|_{\Omega_0}$ for a basis of $L^\perp_g$ and then use
\be\label{Oh}
O_\hl U_{\mu_i} O_\hl^{-1} = \eta_h(\mu_i) U_{\mu_i h}
\ee
to obtain constraints on the matrix $O_\hl$. Schur's lemma guarantees that if we repeat this procedure for enough $\mu_i$, it will fix $O_\hl$ up to an overall phase $\C^*$. Repeating this for all $\hl\in C_g$ gives us the projective representation $\phi_\gl|_{\Omega_0}$.

From the matrices $O_h$ we can read off the cocycle of the representation $c_{\hat g}(\hat h_1,\hat h_2)$. In principle we could now try to reconstruct the obstruction cocycle $\omega$. Instead we will do something simpler.
We are interested only in orbifolds which allow a holomorphic extension, that is for which $\omega$ is trivial. In such a case we can choose $c_{\hat g}(\hat h_1,\hat h_2)=1$. To see if this is compatible with our orbifold, we try to find coboundaries such that $c_{\hat g}(\hat h_1,\hat h_2)=1$. This means that $\omega$ can only be trivial if all projective representations are secretly linear representations.

\subsection{Twining characters}
We now have all the ingredients to compute the twining character
\be
T(\hat g,\hat h;\tau)= \textrm{tr}_{W_{\gl}} \phi_{\hat g}(\hat h) q^{L_0-c/24}
\ee

\begin{thm}
	Let $\hat h$ be a $V_L$-automorphism that commutes with $\hat g$ and $\phi_{\hat g}(\hat h)$ the corresponding action on $W_{\gl}$. Then states of the form
	\be
	\{v_{-n_1}\ldots v_{-n_k} U_{\alpha} e_{\chi}: n_1, \ldots, n_k \in ..., \alpha \in \Lambda, e_{\chi} \in \C[\hat A], v_{-n_i} \in \hat{\mathfrak{h}}_g^{-}\}
	\ee
	contribute to the trace of $\phi_{\hat g}(\hat h)$ only if $\alpha \in \Lambda$ is such that
	\be\label{eq:lambda_h}
	\pi_g(\alpha) = \pi_g(h \alpha),
	\ee
	that is $\pi_g(\alpha) \in \mathfrak{h}^0_g \cap \mathfrak{h}^0_h$ or equivalently $(1-h)\alpha \in L_g^{\perp}$. We denote the sublattice of $\Lambda$ such that Equation \ref{eq:lambda_h} is satisfied by $\Lambda_h$.
\end{thm}
\begin{proof}
	In order for the state to contribute to the character we require that
	\begin{equation}
	[\alpha h] = [\alpha],
	\end{equation}
	where $[\alpha]$ denotes the class of $\alpha$ in $\frac{L}{L_g^{\perp}}$. In other words we have that
	\begin{equation}
	\alpha h = \alpha + \beta^{\perp},
	\end{equation}
	for some $\beta^{\perp} \in L_g^{\perp}$.
\end{proof}

\begin{thm}
	The lattice contribution to the twining character is given by
	\begin{equation}
	T_L(\hat g,\hat h,\tau) = \sum_{\alpha \in \Lambda_h} q^{\frac{|\alpha_0|^2}{2}}\eta_h(\alpha)^{-1}\epsilon_g(\alpha,\alpha(h-1))B_g(\alpha,\alpha(h-1))\mathrm{Tr}|_{\Omega}\big(U_{\alpha(h-1)}\phi_{\hat g}(\hat h)\big).
	\end{equation}
\end{thm}
\begin{proof}
	For $\alpha \in \Lambda_h$ we find that
	\begin{align}
	\phi_{\hat g}(\hat h)v_{-n_1}\ldots v_{-n_k} U_{\alpha} e_{\chi}
	& = (v h)_{-n_1}\ldots (v h)_{-n_k} \eta_h(\alpha)^{-1}U_{\alpha h} \phi_{\hat g}(\hat h)e_{\chi} \\
	& = (v h)_{-n_1}\ldots (v h)_{-n_k} \eta_h(\alpha)^{-1}\epsilon_g(\alpha,\alpha(h-1))B_g(\alpha,\alpha(h-1)) U_{\alpha} U_{\alpha(h-1)}\phi_{\hat g}(\hat h)e_{\chi}.
	\end{align}
\end{proof}
\begin{note}
	If $\alpha(h-1) \in L_g^{\perp}$ then $U_{\alpha(h-1)}\phi_{\hat g}(\hat h)e_{\chi} \in \Omega_0$.
\end{note}

\begin{thm}
	Let $\Psi$ be the set of pairs $(i,j)$ such that $\xi_n^i$ and $\xi_m^j$ are simultaneous eigenvalues of $g$ and $h$.

	The Heisenberg contribution to the twining character is given by
	\begin{equation}
	T_H(\hat g,\hat h,\tau) = q^{\rho_g} \prod_{(i,j) \in \Psi}\prod_{k=0}^{\infty} \frac{1}{1-\xi_m^j q^{k+\frac{i}{n}}},
	\end{equation}
	where the conformal weight $\rho_g$ is given by
	\begin{equation}
		\rho_g = \sum_j \frac{j(n-j)}{4n^2}\mathrm{dim}(\mathfrak{h}_j),
	\end{equation}
	where $\mathfrak{h}_j$ is the eigenspace of $g$ to the eigenvalue $\xi_n^{-j}$.
\end{thm}

\begin{proof}
	Let $v \in \mathfrak{h}$ be a simultaneous eigenvector of $g$ and $h$ with eigenvalues $\xi_n^i$ and $\xi_m^j$, respectively. Then by Definition \ref{def:heisenberg} the weights of the operators corresponding to $v$ in $\hat{\mathfrak{h}}_g$ are given by $\frac{i}{n} + \Z_{\geq 0}$. Then the contribution to the trace of $h$ due to descendants of $v$ is given by
	\begin{align}
		T_v(\hat g,\hat h;\tau) & = q^{\rho_g}(1 + \xi_m^j q^{\frac{i}{n}} + \xi_m^{2j} q^{\frac{2i}{n}} + \ldots)(1 + \xi_m^j q^{1 + \frac{i}{n}} + \xi_m^{2j} q^{2+\frac{2i}{n}} + \ldots)\ldots \\
		& = \prod_{k=0}^{\infty} \frac{1}{1-\xi_m^j q^{k+\frac{i}{n}}}.
	\end{align}
	The desired result follows by multiplying over all simultaneous eigenvectors.
\end{proof}

The character of $\phi_{\hat g}(\hat h)$ on the twisted module $W_{\gl}$ is then given by
\begin{equation}
	T(\hat g, \hat h; \tau) = T_H(\hat g,\hat h,\tau)	T_L(\hat g,\hat h,\tau).
\end{equation}

\section{Warm-up: Orbifolds of $E8$}\label{s:E8}

As a warm-up, let us discuss orbifolds of the $E8$ lattice VOA. This is the smallest even unimodular lattice, which makes our computation easier.
This case is strongly constrained, since the only holomorphic VOA of central charge 8 is the $E8$ lattice VOA. Any orbifold thus only has two possible outcomes: either $\omega$ is not trivial, so that there is no holomorphic extension, or we recover the original unorbifolded VOA. We will give an example for both possibilities.

Investigating all $62092$ conjugacy classes of subgroups of $\mathrm{Aut}(L_{E8})$, we find that under standard lifts only $14$ are non-anomalous of which $10$ are cyclic. The non-cyclic ones are isomorphic to $\Z_3 \times \Z_3$, $\Z_5 \times \Z_5$, $\Z_3\times\Z_6$ and $\Z_5\times\Z_{10}$. We will investigate the smallest one of these as well as anomalous groups isomorphic to $S_3$ and $\Z_2 \times \Z_2$.

\subsection{$S_3$}
First consider an orbifold of the $E8$ lattice VOA by a group $S_3=\Z_3\rtimes_{-1}\Z_2$ generated by elements $s$ and $t$ of cycle types $C_s = 2^4$ and $C_t = 1^23^2$ (see equation \ref{eq:cycle_type} for a definition of the cycle type).  $S_3$ has three conjugacy classes, $[e]$, $[s]$, $[t]$, containing elements of order 1,2, and 3 respectively.
From (\ref{cohompeven}) we have
$H^3(S_3,U(1))=\Z_6$.
The cyclic subgroups generated by $s$ and $t$ have types $r_{s}=1$ and $r_{t}=2$, from which it follows that $\omega$ is non-trivial. We can therefore not obtain a holomorphic orbifold. However, we can still obtain the twining characters for $S_3$
\begin{align}
T(e,e,\tau)  &= \frac{\theta_{\E}(\tau)}{\eta(\tau)^8} &
T(e,s,\tau)  &= \frac{\theta_{D_4}(\tau)}{\eta(2\tau)^4} &
T(e,t,\tau)  &= \frac{\theta_{A_2}(\tau)^2}{\eta(\tau)^2\eta(3\tau)^2} \\
T(s,e,\tau) &= 2\frac{\theta_{D_4}(\frac{\tau}{2})}{\eta(\frac{\tau}{2})^4} &
T(s,s,\tau)  &= 2\xi_{3}\frac{\theta_{D_4}(\frac{\tau+1}{2})}{\eta(\frac{\tau+1}{2})^4} & \\
T(t,e,\tau) &= \frac{\theta_{A_2}(\frac{\tau}{3})^2}{\eta(\tau)^2\eta(\frac{\tau}{3})^2} &
T(t,t,\tau)  &=
\xi_3\frac{\theta_{A_2}(\frac{\tau+1}{3})^2}{\eta(\tau+1)^2\eta(\frac{\tau+1}{3})^2} &
T(t,t^2,\tau)  &= \xi_3^2\frac{\theta_{A_2}(\frac{\tau+2}{3})^2}{\eta(\tau+2)^2\eta(\frac{\tau+2}{3})^2}
\end{align}

Let us now write down the characters of the irreducible modules.
There are $3$ irreducible representation of $S_3$: The trivial, signum and $2$-dimensional standard representation. The other centralizers have the usual cyclic characters. Using the irreducible characters with the 2-cocycles descended from $\omega$, we find
\begin{align}
\chi_{e,0}(\tau) & = \frac{T(e,e,\tau) + 3T(e,s,\tau) + 2T(e,t,\tau)}{6} \\
\chi_{e,\mathrm{sgn}}(\tau) & = \frac{T(e,e,\tau) - 3T(e,s,\tau) + 2T(e,t,\tau)}{6} \\
\chi_{e,2}(\tau) & = \frac{T(e,e,\tau) - T(e,t,\tau)}{3} \\
\chi_{s,0}(\tau)& = \frac{T(s,e,\tau)  - \xi_4 T(s,s,\tau)}{2}\\
\chi_{s,1}(\tau)& = \frac{T(s,e,\tau) + \xi_4 T(s,s,\tau)}{2}\\
\chi_{t,0}(\tau)& = \frac{T(t,e,\tau) + \xi_9^{-2} T(t,t,\tau) + \xi_9^{-4} T(t,t^2,\tau)}{3}\\
\chi_{t,1}(\tau)& = \frac{T(t,e,\tau) + \xi_9 T(t,t,\tau) + \xi_9^2 T(t,t^2,\tau)}{3}\\
\chi_{t,2}(\tau)& = \frac{T(t,e,\tau) + \xi_9^4 T(t,t,\tau) + \xi_9^{-1}T(t,t^2,\tau)}{3},
\end{align}
where $\xi_n$ denotes a primitive $n$-th root of unity.
Using the ordering above, the transformation matrices under $T$ and $S$ agree with the ones obtained in \cite{Dijkgraaf:1989hb,MR1770077 } with $p=1$.

\subsection{$\Z_2\times \Z_2$}
Next we consider an orbifold of the $E8$ lattice VOA by a group $\Z_2 \times \Z_2$ generated by elements $g$ and $h$ both of cyclic type $C_g = C_h = 2^4$.
The types of the cyclic subgroups generated by $g, h$ and $gh$ are given $r_g = r_h = 1$ and $r_{gh} = 0$, respectively. Thus by theorem \ref{H3ZpZp} the cohomological twist $\omega$ is non-trivial.
We can of course still write down the twisted twining characters.
There are 16 of them, which form the following orbits under $\mathrm{SL}(2,\Z)$-transformations:
\begin{multline}
\{T(e,e;\tau) \}\ , \{T(g,e;\tau), T(e,g;\tau), T(g,g;\tau) \}\ ,
\{T(h,e;\tau), T(e,h;\tau), T(h,h;\tau) \}\ , \\
\{T(gh,e;\tau), T(e,gh;\tau), T(gh,gh;\tau) \}\ ,
\{T(g,h;\tau), T(g,gh;\tau), T(h,g;\tau), T(h,gh;\tau), T(gh,g), T(gh,h)  \}.
\end{multline}
Of those orbits, all but the last one intersect the untwisted sector and can be computed using cyclic orbifolds.

For the last orbit however, we need to use our full algorithm: that is, we construct the $g$ and $h$-twisted modules, compute the projective representations $\phi_{g,h}(\cdot)$, and compute the twining characters as described in sections~\ref{s:twistedmodules} and \ref{s:chars}.
By matching the first few coefficients of the expressions we compute, we find experimentally that
\be
T(g,h;\tau) = T(h,g;\tau)= 2\left(\frac{\eta(\tau)^2}{\eta(\tau/2)\eta(2\tau)}\right)^4
\ee
and
\be
T(g,gh;\tau) = T(h,gh)= e^{2\pi i/12}2
\left(\frac{\eta(\tau)^2}{\eta(\tau/2+1/2)\eta(2\tau)}\right)^4\ .
\ee
The product $gh$ however is not a standard lift, so that we leave the construction of the $gh$-twisted module and its characters future work. We conjecture that the characters are given by
\be
T(gh,g;\tau)= T(gh,h;\tau)= e^{-2\pi i/6} \left(\frac{\eta(\tau)^2}{\eta(\tau/2-1/2)\eta(\tau/2)}\right)^4\ .
\ee

\subsection{$\Z_3\times \Z_3$}
Now consider the $\Z_3\times \Z_3$ orbifold generated by two elements $g$ and $h$ of cycle type $1^{-1}3^3$.  We find that all cyclic subgroups have type 0 and hence
\be
\omega = 1\ .
\ee
Again, this orbifold is not
effectively cyclic, and we have $H^2(\Z_3\times \Z_3)=\Z_3\neq 0$, which allows for discrete torsion. This seems to pose a bit of a puzzle: even though discrete torsion seems to allow for different holomorphic extentions, we know that there is only one holomorphic VOA of central charge 8. The resolution to this is  that the contribution of the characters that is controlled by discrete torsion vanishes.

As before,  we can arrange all 81 twining characters into orbits of $\mathrm{SL}(2,\Z)$, giving 6 orbits, one of which is simply $\{T(e,e;\tau)\}$. Of the remaining ones, 4 have length 8 and intersect the untwisted sector, and corresp to the orbifolds by $4$ cyclic subgroups of $\Z_3\times\Z_3$. The representatives for these four orbits $T(e,g;\tau)$, $T(e,h;\tau)$, $T(e,gh;\tau)$ and $T(e,g^2h;\tau)$ are all equal and given by
\begin{equation}
T(e,g;\tau) = T(e,h;\tau) = T(e,gh;\tau) = T(e,g^2h;\tau) = \frac{\theta_{A_2}(\tau)}{\eta(\tau)^{-1}\eta(3\tau)^3}
\end{equation}
 The fifth orbit with representative $T(g,h;\tau)$ has length 48 and does not intersect the untwisted sector. We therefore again need to use our full algorithm to compute twisted modules and their twining characters explicitly. It turns out however that all twining characters are constant. The reason for this is that for all pairs $(g_1,g_2)$ that appear in this orbit, the lattice $\Lambda$ complementary to $L_{g_1}^\perp$ does not have any $g_2$-invariant vectors, so that the generalized theta function is a constant. Moreover the $\eta$ functions turn out to give a constant, so that the overall character is indeed in turn constant. Acting with $T$ on such a character thus simply introduces a third root of unity. When we decompose the orbit into 16 $T$ orbits, each of them is a sum of three roots of unity, which vanishes. The upshot is thus that this orbit, whose contribution is in principle controlled by discrete torsion, never contributes to the character of $V^{orb(G)}$.

There is in fact a quicker way to see this: since this orbit does not intersect the untwisted module, and the conformal weights of the twisted are all strictly positive, no term $q^{-1/3}$ can appear, but only terms with higher powers of $q$. There is no modular function with the right multiplier system and only higher powers, so that it follows that the sum over the orbit necessarily vanishes.

\section{Examples with $c=48$ and $c=72$}\label{s:lattice}
Let us now describe the groups $G$ with respect we orbifold. Extremal lattices in $d=48$ and $72$ and their automorphisms were constructed in \cite{MR1662447,MR1489922,MR2999133,MR3225314}, and their generators can found in \cite{latticedb}. Using MAGMA \cite{MR1484478}, we extracted all subgroups of $Aut(L)$ of the form $\Z_q \rtimes_{\phi} \Z_p$ as in definition~\ref{defGqp} with $q,p$ coprime. We then restrict to groups for which no lifted element $\gl$ had its order doubled. By the procedure in section~\ref{ss:effcyclift}, we can lift them to automorphism groups of $V_L$ which were isomorphic to $\Z_q \rtimes_{\phi} \Z_p$. We then only keep the ones which had $[\omega]=1$. For these we use theorem~\ref{effcycformula} to compute the characters of $V^{orb(G)}$.

In the following tables we list the orbifold characters for all such groups. Different groups can give the same character; in these cases we did not check if the orbifold VOAs just happen to have the same character, or if they are actually isomorphic. In the last column we give the polar and constant part of the character, from which of course the entire character can be reconstructed. In columns two and four we give the cycle type of the cyclic subgroups $\Z_q$ and $\Z_p$ written as $\prod_{t|n}t^{b_t}$ where $b_t$ can be read off from the characteristic polynomial of the generator $g$ of $\Z_n$,
\be{\label{eq:cycle_type}}
\mathcal{P}_g(x) = \prod_{t|n}(x^t-1)^{b_t}\ .
\ee

\subsection{$\Z_q \rtimes_{\phi} \Z_p$-Orbifolds for $d=48$}
\subsubsection{$P_{48p}$}

\begin{center}
	\begin{longtable}{llllll}
		\toprule
		$q$ & $\Z_q$     & $p$ & $\Z_p$ & $\phi$ & $\text{ch}_{V^{\text{orb}(\Z_q \rtimes_\phi\Z_p)}}(q)$ \\ \midrule \endhead
		$23$ & $1^{2}23^{2}$ & $132$ & $2^{1}4^{-1}6^{-1}12^{1}22^{1}44^{1}66^{-1}132^{1}$ & 5 & $q^{-2} + 120 + \mathcal{O}(q)$ \\
		$23$ & $1^{2}23^{2}$ & $66$ & $1^{2}2^{-2}3^{-2}6^{2}11^{2}22^{-2}33^{-2}66^{2}$ & 2 & $q^{-2} + 168 + \mathcal{O}(q)$ \\
		$23$ & $1^{2}23^{2}$ & $44$ & $2^{-2}4^{2}22^{-2}44^{2}$ & 5 & $q^{-2} + 48 + \mathcal{O}(q)$ \\
		$23$ & $1^{2}23^{2}$ & $33$ & $1^{-2}3^{2}11^{-2}33^{2}$ & 2 & $q^{-2} + 48 + \mathcal{O}(q)$ \\
		$23$ & $1^{2}23^{2}$ & $22$ & $1^{-4}2^{4}11^{-4}22^{4}$ & 2 & $q^{-2} + 168 + \mathcal{O}(q)$ \\
		$23$ & $1^{2}23^{2}$ & $11$ & $1^{4}11^{4}$ & 2 & $q^{-2} + 4q^{-1} 168 + \mathcal{O}(q)$ \\
		$23$ & $1^{2}23^{2}$ & $12$ & $2^{12}4^{-12}6^{-12}12^{12}$ & -1 & $q^{-2} + 120 + \mathcal{O}(q)$ \\
		$23$ & $1^{2}23^{2}$ & $4$ & $2^{-24}4^{24}$ & -1 & $q^{-2} + 48 + \mathcal{O}(q)$ \\
		$3$ & $1^{-24}3^{24}$ & $4$ & $2^{-24}4^{24}$ & -1 & $q^{-2} + 576 + \mathcal{O}(q)$ \\
		$11$ & $1^{4}11^{4}$ & $4$ & $2^{-24}4^{24}$ & -1 & $q^{-2} + 96 + \mathcal{O}(q)$ \\
		$11$ & $1^{4}11^{4}$ & $12$ & $2^{12}4^{-12}6^{-12}12^{12}$ & -1 & $q^{-2} + 168 + \mathcal{O}(q)$ \\
		$33$ & $1^{-2}3^{2}11^{-2}33^{2}$ & $4$ & $2^{-24}4^{24}$ & -1 & $q^{-2} + 96 + \mathcal{O}(q)$ \\ \bottomrule
	\end{longtable}
\end{center}

\subsubsection{$P_{48q}$}

\begin{center}
	\begin{longtable}{llllll}
		\toprule
$q$ & $\Z_q$     & $p$ & $\Z_p$ & $\phi$ & $\text{ch}_{V^{\text{orb}(\Z_q \rtimes_\phi\Z_p)}}(q)$ \\ \midrule \endhead
		$47$ & $1^{1}47^{1}$ & $46$ & $1^{-2}2^{2}23^{-2}46^{2}$ & 2 & $q^{-2} + 72 + \mathcal{O}(q)$ \\
		$47$ & $1^{1}47^{1}$ & $23$ & $1^{2}23^{2}$ & 2 & $q^{-2} + 2q^{-1} + 120 + \mathcal{O}(q)$ \\
		$23$ & $1^{2}23^{2}$ & $4$ & $2^{-24}4^{24}$ & -1 & $q^{-2} + 48 + \mathcal{O}(q)$ \\ \bottomrule
	\end{longtable}
\end{center}

\subsubsection{$P_{48m}$}

\begin{center}
	\begin{longtable}{llllll}
		\toprule
$q$ & $\Z_q$     & $p$ & $\Z_p$ & $\phi$ & $\text{ch}_{V^{\text{orb}(\Z_q \rtimes_\phi\Z_p)}}(q)$ \\ \midrule \endhead
		$3$ & $1^{-24}3^{24}$ & $20$ & $2^{6}4^{-6}10^{-6}20^{6}$ & -1 & $q^{-2} + 264 + \mathcal{O}(q)$ \\
		$3$ & $1^{-24}3^{24}$ & $4$ & $2^{-24}4^{24}$ & -1 & $q^{-2} + 576 + \mathcal{O}(q)$ \\
		$15$ & $1^{-4}3^{4}5^{-4}15^{4}$ & $4$ & $2^{-24}4^{24}$ & -1 & $q^{-2} + 192 + \mathcal{O}(q)$ \\
		$5$ & $1^{-12}5^{12}$ & $12$ & $2^{12}4^{-12}6^{-12}12^{12}$ & -1 & $q^{-2} + 360 + \mathcal{O}(q)$ \\
		$5$ & $1^{8}5^{8}$ & $12$ & $2^{12}4^{-12}6^{-12}12^{12}$ & -1 & $q^{-2} + 264 + \mathcal{O}(q)$ \\
		$5$ & $1^{-2}5^{10}$ & $12$ & $2^{12}4^{-12}6^{-12}12^{12}$ & -1 & $q^{-2} + 216 + \mathcal{O}(q)$ \\
		$5$ & $1^{-2}5^{10}$ & $12$ & $2^{12}4^{-12}6^{-12}12^{12}$ & -1 & $q^{-2} + 216 + \mathcal{O}(q)$ \\
		$5$ & $1^{8}5^{8}$ & $4$ & $2^{-24}4^{24}$ & -1 & $q^{-2} + 192 + \mathcal{O}(q)$ \\
		$5$ & $1^{8}5^{8}$ & $4$ & $2^{-24}4^{24}$ & -1 & $q^{-2} + 192 + \mathcal{O}(q)$ \\
		$5$ & $1^{-12}5^{12}$ & $4$ & $2^{-24}4^{24}$ & -1 & $q^{-2} + 288 + \mathcal{O}(q)$ \\
		$5$ & $1^{-2}5^{10}$ & $4$ & $2^{-24}4^{24}$ & -1 & $q^{-2} + 144 + \mathcal{O}(q)$ \\
		$5$ & $1^{-2}5^{10}$ & $4$ & $2^{-24}4^{24}$ & -1 & $q^{-2} + 144 + \mathcal{O}(q)$ \\
		\bottomrule
	\end{longtable}
\end{center}

\subsubsection{$P_{48n}$}

\begin{center}
	\begin{longtable}{llllll}
		\toprule
$q$ & $\Z_q$     & $p$ & $\Z_p$ & $\phi$ & $\text{ch}_{V^{\text{orb}(\Z_q \rtimes_\phi\Z_p)}}(q)$ \\ \midrule \endhead
		$5$ & $1^{-12}5^{12}$ & $52$ & $2^{2}4^{-2}26^{-2}52^{2}$ & -1 & $q^{-2} + 120 + \mathcal{O}(q)$ \\
		$13$ & $1^{-4}13^{4}$ & $20$ & $2^{6}4^{-6}10^{-6}20^{6}$ & -1 & $q^{-2} + 120 + \mathcal{O}(q)$ \\
		$13$ & $1^{-4}13^{4}$ & $12$ & $2^{12}4^{-12}6^{-12}12^{12}$ & -1 & $q^{-2} + 168 + \mathcal{O}(q)$ \\
		$13$ & $1^{-4}13^{4}$ & $12$ & $2^{12}4^{-12}6^{-12}12^{12}$ & -1 & $q^{-2} + 168 + \mathcal{O}(q)$ \\
		$39$ & $1^{2}3^{-2}13^{-2}39^{2}$ & $4$ & $2^{-24}4^{24}$ & -1 & $q^{-2} + 96 + \mathcal{O}(q)$ \\
		$3$ & $1^{-24}3^{24}$ & $52$ & $2^{2}4^{-2}26^{-2}52^{2}$ & -1 & $q^{-2} + 168 + \mathcal{O}(q)$ \\
		$5$ & $1^{-12}5^{12}$ & $28$ & $2^{4}4^{-4}14^{-4}28^{4}$ & -1 & $q^{-2} + 168 + \mathcal{O}(q)$ \\
		$5$ & $1^{-12}5^{12}$ & $28$ & $2^{4}4^{-4}14^{-4}28^{4}$ & -1 & $q^{-2} + 168 + \mathcal{O}(q)$ \\
		$35$ & $1^{2}5^{-2}7^{-2}35^{2}$ & $4$ & $2^{-24}4^{24}$ & -1 & $q^{-2} + 96 + \mathcal{O}(q)$ \\
		$7$ & $1^{-8}7^{8}$ & $20$ & $2^{6}4^{-6}10^{-6}20^{6}$ & -1 & $q^{-2} + 168 + \mathcal{O}(q)$ \\
		$7$ & $1^{-8}7^{8}$ & $12$ & $2^{12}4^{-12}6^{-12}12^{12}$ & -1 & $q^{-2} + 264 + \mathcal{O}(q)$ \\
		$7$ & $1^{-8}7^{8}$ & $12$ & $2^{12}4^{-12}6^{-12}12^{12}$ & -1 & $q^{-2} + 264 + \mathcal{O}(q)$ \\
		$7$ & $1^{-8}7^{8}$ & $12$ & $2^{12}4^{-12}6^{-12}12^{12}$ & -1 & $q^{-2} + 264 + \mathcal{O}(q)$ \\
		$21$ & $1^{4}3^{-4}7^{-4}21^{4}$ & $4$ & $2^{-24}4^{24}$ & -1 & $q^{-2} + 192 + \mathcal{O}(q)$ \\
		$21$ & $1^{4}3^{-4}7^{-4}21^{4}$ & $4$ & $2^{-24}4^{24}$ & -1 & $q^{-2} + 192 + \mathcal{O}(q)$ \\
		$21$ & $1^{4}3^{-4}7^{-4}21^{4}$ & $4$ & $2^{-24}4^{24}$ & -1 & $q^{-2} + 192 + \mathcal{O}(q)$ \\
		$3$ & $1^{-24}3^{24}$ & $28$ & $2^{4}4^{-4}14^{-4}28^{4}$ & -1 & $q^{-2} + 264 + \mathcal{O}(q)$ \\
		$3$ & $1^{-24}3^{24}$ & $28$ & $2^{4}4^{-4}14^{-4}28^{4}$ & -1 & $q^{-2} + 264 + \mathcal{O}(q)$ \\
		$3$ & $1^{-24}3^{24}$ & $28$ & $2^{4}4^{-4}14^{-4}28^{4}$ & -1 & $q^{-2} + 264 + \mathcal{O}(q)$ \\
		$3$ & $1^{-24}3^{24}$ & $26$ & $2^{-2}26^{2}$ & -1 & $q^{-2} + 84 + \mathcal{O}(q)$ \\
		$7$ & $1^{-8}7^{8}$ & $10$ & $2^{-6}10^{6}$ & -1 & $q^{-2} + 84 + \mathcal{O}(q)$ \\
		$13$ & $1^{-4}13^{4}$ & $4$ & $2^{24}4^{-24}$ & -1 & $q^{-2} + 96 + \mathcal{O}(q)$ \\
		$7$ & $1^{-8}7^{8}$ & $6$ & $2^{-12}6^{12}$ & -1 & $q^{-2} + 132 + \mathcal{O}(q)$ \\
		$3$ & $1^{-24}3^{24}$ & $14$ & $2^{-4}14^{4}$ & -1 & $q^{-2} + 132 + \mathcal{O}(q)$ \\
		$7$ & $1^{-8}7^{8}$ & $4$ & $2^{-24}4^{24}$ & -1 & $q^{-2} + 192 + \mathcal{O}(q)$ \\
		$7$ & $1^{-8}7^{8}$ & $4$ & $2^{-24}4^{24}$ & -1 & $q^{-2} + 192 + \mathcal{O}(q)$ \\
		$7$ & $1^{-8}7^{8}$ & $4$ & $2^{-24}4^{24}$ & -1 & $q^{-2} + 192 + \mathcal{O}(q)$ \\
		$7$ & $1^{-8}7^{8}$ & $4$ & $2^{-24}4^{24}$ & -1 & $q^{-2} + 192 + \mathcal{O}(q)$ \\
		$5$ & $1^{-12}5^{12}$ & $4$ & $2^{-24}4^{24}$ & -1 & $q^{-2} + 288 + \mathcal{O}(q)$ \\
		$5$ & $1^{-12}5^{12}$ & $4$ & $2^{-24}4^{24}$ & -1 & $q^{-2} + 288 + \mathcal{O}(q)$ \\
		$3$ & $1^{-24}3^{24}$ & $4$ & $2^{-24}4^{24}$ & -1 & $q^{-2} + 576 + \mathcal{O}(q)$ \\
		\bottomrule
	\end{longtable}
\end{center}

\subsection{$\Z_q \rtimes \Z_p$-Orbifolds for $d=72$}

\begin{center}
	\begin{longtable}{llllll}
		\toprule
$q$ & $\Z_q$     & $p$ & $\Z_p$ & $\phi$ & $\text{ch}_{V^{\text{orb}(\Z_q \rtimes_\phi\Z_p)}}(q)$ \\ \midrule \endhead
		$7$ & $1^{-12}7^{12}$ & $78$ & $3^{2}6^{-2}39^{-2}78^{2}$ & 2 & $q^{-3} + 12q^{-1} + 200 + \mathcal{O}(q)$ \\
		$7$ & $1^{-12}7^{12}$ & $30$ & $3^{-4}6^{4}15^{-4}30^{4}$ & 2 & $q^{-3} + 32q^{-1} + 304 + \mathcal{O}(q)$ \\
		$7$ & $1^{-12}7^{12}$ & $30$ & $3^{6}6^{-6}15^{-6}30^{6}$ & 2 & $q^{-3} + 36q^{-1} + 568 + \mathcal{O}(q)$ \\
		$7$ & $1^{-12}7^{12}$ & $39$ & $3^{-2}39^{2}$ & 2 & $q^{-3} + 12q^{-1} + 328 + \mathcal{O}(q)$ \\
		$7$ & $1^{-12}7^{12}$ & $15$ & $3^{4}15^{4}$ & 2 & $q^{-3} + 32q^{-1} + 848 + \mathcal{O}(q)$ \\
		$7$ & $1^{-12}7^{12}$ & $15$ & $3^{-6}15^{6}$ & 2 & $q^{-3} + 36q^{-1} + 824 + \mathcal{O}(q)$ \\
		$3$ & $3^{24}$ & $26$ & $1^{-2}2^{-2}13^{2}26^{2}$ & -1 & $q^{-3} + 48q^{-1} + 1344 + \mathcal{O}(q)$ \\
		$3$ & $3^{24}$ & $26$ & $1^{2}2^{-4}13^{-2}26^{4}$ & -1 & $q^{-3} + 48q^{-1} + 1272 + \mathcal{O}(q)$ \\
		$7$ & $1^{-12}7^{12}$ & $6$ & $3^{-24}6^{24}$ & 2 & $q^{-3} + 144q^{-1} + 1040 + \mathcal{O}(q)$ \\
		$7$ & $1^{-12}7^{12}$ & $6$ & $3^{-24}6^{24}$ & 2 & $q^{-3} + 144q^{-1} + 1040 + \mathcal{O}(q)$ \\
		$3$ & $3^{24}$ & $10$ & $1^{4}2^{4}5^{4}10^{4}$ & -1 & $q^{-3} + 8q^{-2
		} + 152q^{-1} + 4320 + \mathcal{O}(q)$ \\
		$3$ & $3^{24}$ & $10$ & $1^{6}2^{-12}5^{-6}10^{12}$ & -1 & $q^{-3} + 144q^{-1} + 3336 + \mathcal{O}(q)$ \\
		$3$ & $3^{24}$ & $10$ & $1^{-6}2^{-6}5^{6}10^{6}$ & -1 & $q^{-3} + 144q^{-1} + 3456 + \mathcal{O}(q)$ \\
		$3$ & $3^{24}$ & $10$ & $1^{-4}2^{8}5^{-4}10^{8}$ & -1 & $q^{-3} + 136q^{-1} + 3360 + \mathcal{O}(q)$ \\
		$7$ & $1^{-12}7^{12}$ & $3$ & $3^{24}$ & 2 & $q^{-3} + 144q^{-1} + 3376 + \mathcal{O}(q)$ \\
		$3$ & $3^{24}$ & $2$ & $1^{-24}2^{48}$ & -1 & $q^{-3} + 600q^{-1} + 14496 + \mathcal{O}(q)$ \\
		$3$ & $3^{24}$ & $2$ & $1^{24}2^{24}$ & -1 & $q^{-3} + 24q^{-2} + 648q^{-1} + 17376 + \mathcal{O}(q)$ \\ \bottomrule
	\end{longtable}
\end{center}

\subsection{Bounds for Abelian orbifolds}

At $d=24$ it turns out that all $71$ holomorphic vertex operator algebras can be constructed as cyclic orbifolds of lattice VOAs. Now it is a reasonable question to ask whether that might also be the case at higher central charge.
Consider a lattice vertex operator algebra at central charge $d$ and let $\mathfrak{h}$ be its heisenberg sub-VOA.
\begin{lem}
For any finite, abelian group $G$ acting orthogonally on the heisenberg VOA $\mathfrak{h}$ the dimensions of the $G$-invariant homogeneous subspaces are bounded from below by
\begin{equation}
	\mathrm{dim}(\mathfrak{h}_{(2)}) \geq \frac{d}{2}
\end{equation}
and
\begin{equation}
	\mathrm{dim}(\mathfrak{h}_{(3)}) \geq d.
\end{equation}
In particular, the bounds are sharp.
\end{lem}
\begin{proof}
Let $\{v_i\}$ denote the set of simultaneous eigenvectors of the elements of $G$.

$\mathfrak{h}_{(2)}$ is spanned by states of the form $(v_i)_{-2}\mathfrak{e}_0$ and $(v_i)_{-1}(v_j)_{-1}\mathfrak{e}_0$. Clearly, none of the former states will be invariant if $G$ contains a fixed-point-free element. Among the latter, elements of the form $(v_i)_{-1}(\overline{v_i})_{-1}\mathfrak{e}_0$ are clearly $G$-invariant yielding at least $\frac{d}{2}$ invariant elements. If $G$ contains an element $g$ with no real eigenvalues and whose characteristic polynomial is squarefree then there are exactly $\frac{d}{2}$ states of this form and they contain all $g$-invariant states at level $2$. Hence the bound is sharp.

$\mathfrak{h}_3$ is spanned by states of the form $(v_i)_{-3}\mathfrak{e}_0$, $(v_i)_{-2}(v_j)_{-1}\mathfrak{e}_0$ and $(v_i)_{-1}(v_j)_{-1}(v_k)_{-1}\mathfrak{e}_0$. By a similar reasoning to above the first two types of states contribute at least a total of $d$ states. A state of the last type is invariant under an element $g$ if the respective eigenvalues satisfy $\lambda_i\lambda_j\lambda_k = 1$. Now let $g$ have even order $n$ and let all its eigenvalues be primitive $n$-th roots of unity. Then the product of any $3$ eigenvalues is an odd power of a primitive $n$-th root of unity and there no invariant states. Hence the bound $\mathrm{dim}(\mathfrak{h}_{(3)}) \geq d$ is sharp.
\end{proof}

Hence we find that for any finite abelian automorphism group of lifted lattice automorphisms the number of low-weight states is bounded from below. In particular, this provides us with an immediate answer to our question:

\begin{cor}
	The vertex operator algebra at central charge $d=72$ with character $\chi_{V}(\tau) = q^{-3} + 12q^{-1} + 200 + \mathcal{O}(q)$ cannot be constructed as a cyclic orbifold of a lattice vertex operator algebra by a lifted lattice automorphism.
\end{cor}

\appendix

\section{Smith normal form and lattice quotients}\label{app:Davydov}
\begin{defn}[Smith normal form]
	Let $L$ be an even, unimodular lattice of rank $n$ and $A$ an $n \times n$-matrix with entries in the integers (or more generally, a principal ideal domain). Then there exist integral, (over the integers) invertible matrices $P$ and $Q$ such that the product
	\begin{equation}\label{eq:snf}
	S = P~A~Q
	\end{equation}
	is a diagonal matrix of the form
	\begin{equation}
	S = \mathrm{diag}(s_1,s_2,...,s_k,0,\ldots,0),
	\end{equation}
	such that $s_i \in \Z^+$ and $s_i | s_{i+1}$ for all $i$ and $\prod_i s_i = \pm \mathrm{det}(A)$.
	The diagonal entries $s_i$ are unique up to sign and are called the \textit{elementary divisors} of $A$.
	$S$ is called the \textit{Smith normal form} of $A$. To put it another way, there exists a basis Let $\{\tilde \alpha_1, \ldots, \tilde \alpha_k, \gamma_1, \ldots, \gamma_{n-k}\}$ and a basis $\{\alpha_1, \ldots, \alpha_k, \delta_1, \ldots, \delta_{n-k}\}$ of $L$ (given by the row vectors of $P$ and $Q^{-1}$ respectively) such that $A$ acts as
		\begin{equation}
	\tilde \alpha_i A = s_i \alpha_i, \quad i = 1, \ldots, k
	\end{equation}
	and
	\begin{equation}
	\gamma_j A = 0, \quad j = 1, \ldots, n-k.
	\end{equation}
\end{defn}

Since the lattice $LA$ is spanned by the basis $\{s_i \alpha_i \}$, it follows immediately:
\begin{cor}\label{thm:lat_quo}
	Let $L$ be a lattice of rank $n$ and let $A$ act on elements of $L$ in the coordinate basis.
	Then the group structure of the quotient $L/LA$ is given by
	\begin{equation}
	L/LA \cong \bigoplus_i^n \Z/s_i\Z,
	\end{equation}
	where $s_i = 0$ for $i > k$.
\end{cor}

\begin{lem}{\label{lem:smith}}
	Let $g$ be an automorphism of $L$ of order $n$, $A = 1-g$ and the matrices $P$, $Q$ and $S$ as in Theorem \ref{thm:lat_quo}.
	Then the fixed-point lattice $L_g$ is given by
	\begin{equation}
	L_g = \mathrm{span}_{\Z}(\gamma_1, \ldots, \gamma_{n-k})
	\end{equation}
	and $L^{\perp}_g$, its orthogonal complement in $L$, by
	\begin{equation}
	L^{\perp}_g = \mathrm{span}_{\Z}(\alpha_1, \ldots, \alpha_k).
	\end{equation}
\end{lem}
\begin{proof}
	Clearly, $\mathrm{span}_{\Z}(\gamma_1, \ldots, \gamma_{n-k}) \subseteq L^g$. So to prove the first statement, we need to show that there is no vector $v \in \mathrm{span}_{\Z}(\tilde \alpha_1, \ldots,\tilde \alpha_k)$
	such that $v(1-g) = 0$. This follows directly from the linear independence of the $\{\alpha_i\}$.

	Clearly, $\mathrm{span}_{\Z}(\alpha_1, \ldots, \alpha_k) \subseteq L^{\perp}$. So to prove the second statement, we need to show that there is no vector $w \in \mathrm{span}_{\Z}(\delta_1, \ldots,\delta_{n-k})$ such that $w \in L^{\perp}$. But because $\mathrm{Rank}(L^{\perp}) = \mathrm{Rank}(L(1-g))$ (Theorem \ref{thm:lat_quo} implies that) for any such vector there has to exist an integer $n_w$ such that $n_w w \in L(1-g)$. This contradicts the linear independence of the elements of $B_Q$. Hence the second statement follows.
\end{proof}

\begin{cor}
	$L^{\perp}_g/L(1-g)$ is generated by the elements $\{[\alpha_1], \ldots, [\alpha_k]\}$ and the element $[\alpha_i]$ has order $s_i$.
 In particular, $L_g^{\perp}/L(1-g)$ is the torsion subgroup of $L/L(1-g)$.
\end{cor}

\section{Lagrangian decomposition and Darboux basis}
Let $A$ be a finite abelian group and $\beta:A \times A \to \C$ an alternative bi-multiplicative non-degenerate form.

\begin{lem}\label{lem:ndeg}
	Let $\{a_1, \ldots, a_k\}$ be a basis of $A$ such that $a_1$ has order $n$. Then there is another basis element $b_1:=a_i$, $i \neq 1$ such that $\beta(a_1,b_1)$ is a primitive $n$-th root of unity.
\end{lem}
\begin{proof}
	Assume $\beta(\alpha_1,\alpha_j)$ is at most an $m$-th root of unity, for some proper divisor $m\mid n$ and for all $j$.
	Then $\beta(m\alpha_1,\alpha_j)=1$ for all $j$ and hence $\beta$ is degenerate. A contradiction.
\end{proof}
\begin{note}
	This also implies that $ord(b_1) = ord(a_1)$.
\end{note}

\begin{thm}\label{thm:darboux1}
	$A$ admits a basis $\{a_1, \ldots, a_k, b_1, \ldots, b_k\}$ such that
	\begin{equation}
	o(a_i) = o(b_i) = n_i,
	\end{equation}
	where $n_i|n_{i+1}$ and such that
	\begin{equation}
	\beta(a_i,b_j) = \delta_{i,j} \xi_{n_i}^{l_i},
	\end{equation}
	where $\mathrm{gcd}(n_i,l_i) = 1$ and
	\begin{equation}
	\beta(a_i,a_j) = \beta(b_i,b_j) = 1.
	\end{equation}
\end{thm}
\begin{proof}
	By the fundamental theorem of finitely generated abelian groups there exists a basis $\{a_1, \ldots, a_r\}$ of $A$ such $ord(a_i) = n_i$,$n_r = n_k$ and $n_i\mid n_{i+1}$.
	Then by Lemma \ref{lem:ndeg} there exists an $a_j=:b_r$ such that
	\begin{equation}
	\beta(a_r,b_r) = \chi.
	\end{equation}
	where $\chi$ is a primitive $n$-th root of unity. Hence for any basis element $a_i$, $i \neq r,j$ there are positive integers $k_{i,r}$ and $k_{i,j}$ such that
	\begin{equation}
	\beta(a_i,a_r) = \chi^{k_{i,r}}
	\end{equation}
	and
	\begin{equation}
	\beta(a_i, a_j) = \chi^{k_{i,j}}.
	\end{equation}
	Now set
	\begin{equation}
	\tilde a_i = a_i + k_{i,r} a_j - k_{i,j} a_r.
	\end{equation}
	Note that $k_{i,r} a_j$ and $k_{i,j} a_r$ and hence also $\tilde \alpha_i$ are elements of order $n_i$.
	Then $\{\tilde a_1, \ldots,\tilde a_{r-1}, a_r,b_r\}$ is a basis of $A$. The theorem follows by induction. (To do: iron the end.)
\end{proof}

\begin{lem}
	Let $\FA = \frac{L_g^{\perp}}{L(1-g)}$ and $\beta = C$. Then there is a basis
	$\{\alpha_1, \ldots, \alpha_k, \beta_1, \ldots, \beta_k\}$ of $L^{\perp}$ such that
	$\{[\alpha_1], \ldots, [\alpha_k], [\beta_1], \ldots, [\beta_k]\}$ is a basis of $N$ that
	satisfies the conditions from Theorem \ref{thm:darboux1}.
\end{lem}
\begin{proof}
	The orthogonalisation in proof of  Theorem \ref{thm:darboux1} describes valid change of basis of $L^{\perp}$.
\end{proof}

\section{Lifting theory}

\subsection{Construction of the Lifting map}\label{app:lift}
Let $L$ be an even self-dual lattice with inner product $\langle.|.\rangle$ and gram matrix $G$ and let $g$ an automorphism of $L$ of order $n$.
Let $\mathfrak{h} = L \otimes_{\Z} \C$ denote the complex ambient vector space of $L$.
Let $(.,.)$ denote the standard euclidean inner product w.r.t. coordinate vectors in $L$ such that $\langle \alpha|\beta \rangle = (\alpha G, \beta) = (\alpha, \beta G)$.
Note that, by convention, vectors are row vectors and matrices act to the left.

To define lattice vertex operators, we need the central extension $\hat{L}$ of $L$ (as a free abelian group) defined by
\[
\eL_{\alpha}\eL_{\beta} = \epsilon(\alpha,\beta)\eL_{\alpha+\beta}, \quad \alpha,\beta \in L, \quad \eL_{\alpha},\eL_{\beta},\eL_{\alpha+\beta} \in \hat{L},
\]
where $\epsilon:L \times L \to \C^*$ is a $2$-cocycle such that \be\label{eq:comm}\frac{\epsilon(\alpha,\beta)}{\epsilon(\beta,\alpha)} = (-1)^{\langle \alpha,\beta \rangle}\ee.
More precisely,
\begin{defn}{\label{def:eps}}
	The 2-cocycle $\epsilon:L\times L \to \{\pm1\}$ is a bimultiplicative function 		satisfying
	\be\label{eq:eps_1}
	\epsilon(\alpha,\alpha) = (-1)^{\frac{|\alpha|^2}{2}}
	\ee
	and then by bimultiplicativity
	\be\label{eq:eps_2}
	\epsilon(\alpha,\beta)\epsilon(\beta,\alpha) = (-1)^{\langle \alpha|\beta \rangle} = (-1)^{(\alpha G, \beta)}.
	\ee
	The skew fixes in particular the cohomology class of $\epsilon$.
\end{defn}

\begin{defn}
	The lift of $g$ to a VOA automorphism $\hat g$ is determined by the function $\eta_g:L \to \C^*$ satisfying
	\be\label{eq:eta_def}
	\eta_g(\alpha)\eta_g(\beta)\epsilon(\alpha,\beta) = \eta_g(\alpha + \beta) \epsilon(\alpha g,\beta g).
	\ee
	We use the convention that lattice states $\eL_{\alpha} \in V_L$ transform as
	\be
	\hat{g}\eL_{\alpha} = \eta_g(\alpha)^{-1}\eL_{\alpha g}.
	\ee
\end{defn}
Let us now construct a lift of an element $g$. First, we construct the cocycle $\epsilon$ explicitly:
\begin{defn}
	Let $M$ be a matrix. Then let $\overline M$ denote the lower triangular matrix such that
	\begin{equation}
	\overline M_{i,j} = M_{i,j}, \quad \text{if } i > j
	\end{equation}
	and
	\begin{equation}
	\overline M_{i,i} = \frac{1}{2}M_{i,i}, \quad \text{for all } i.
	\end{equation}
	In particular, if $M$ is symmetric then
	\begin{equation}
	\overline M + \overline M^{\mathrm{T}} = M.
	\end{equation}
\end{defn}

\begin{lem}\label{lem:eps}
	A possible solution for equations \ref{eq:eps_1} and \ref{eq:eps_2} is given by
	\begin{equation}\label{eq:eps_sol}
	\epsilon(\alpha,\beta) = (-1)^{(\alpha \overline G,\beta)}.
	\end{equation}
\end{lem}
\begin{proof}
	Immediate.
\end{proof}
Now we can use this to construct the map
\begin{thm}\label{thm:etaexpression}
	Let $\epsilon(\alpha,\beta)$ be given by Equation \ref{eq:eps_sol}. Let $\{v_i\}$ be the basis of $L$ and $\alpha := \sum_i n_i v_i$ a vector in $L$. Then  the general solution of Equation \ref{eq:eta_def} is given by
	\begin{align}
	\eta_g(\sum_i n_i v_i)& = \prod_{i}\eta_g(v_i)^{n_i}\bigg(\frac{\epsilon(v_i,v_i)}{\epsilon(v_ig,v_ig)}\bigg)^{\frac{n_i(n_i-1)}{2}}\prod_{i > j} \bigg(\frac{\epsilon(v_i,v_j)}{\epsilon(v_ig,v_jg)}\bigg)^{n_in_j}\label{eq:eta_rec} \\
	& = (-1)^{(\alpha\overline{(\overline{G}-g\overline{G}g^{\mathrm{T}})},\alpha)}\prod_{i}\bigg(\eta_g(v_i)\bigg(\frac{\epsilon(v_i,v_i)}{\epsilon(v_ig,v_ig)}\bigg)^{-\frac{1}{2}}\bigg)^{n_i},
	\end{align}
	where we are free to choose the $\eta_g(v_i)$.
\end{thm}
\begin{proof}
	Let $M = \overline{G}-g\overline{G}g^{\mathrm{T}}$. Note that $M$ is antisymmetric, as
	\begin{align}
	\overline{G}-g\overline{G}g^{\mathrm{T}} + (\overline{G}-g\overline{G}g^{\mathrm{T}})^{\mathrm{T}} & =
	\overline{G}-g\overline{G}g^{\mathrm{T}} + 	\overline{G}^{\mathrm{T}}-g\overline{G}^{\mathrm{T}}g^{\mathrm{T}}  \\
	& = (\overline{G} + \overline{G}^{\mathrm{T}}) - g(\overline{G}+\overline{G}^{\mathrm{T}})g^{\mathrm{T}} \\
	& = G - gGg^{\mathrm{T}} = 0.
	\end{align}

	Let $\alpha = \sum_i n_i v_i$ and $\beta = \sum_i m_i v_i$. Then
	\begin{align}
	\eta_g(\alpha)\eta_g(\beta) & \frac{\epsilon(\alpha,\beta)}{\epsilon(\alpha g,\beta g)} =
	(-1)^{(\alpha \overline{M},\alpha)+(\beta \overline M, \beta) + (\alpha M,\beta)}\prod_{i}\bigg(\eta_g(v_i)\bigg(\frac{\epsilon(v_i,v_i)}{\epsilon(v_ig,v_ig)}\bigg)^{-\frac{1}{2}}\bigg)^{n_i+m_i}\label{ln:eta_pf1} \\
	& = (-1)^{(\alpha \overline{M},\alpha)+(\beta \overline M, \beta) +
		(\alpha \overline M,\beta) - (\alpha \overline{M}^{\mathrm{T}},\beta)}\prod_{i}\bigg(\eta_g(v_i)\bigg(\frac{\epsilon(v_i,v_i)}{\epsilon(v_ig,v_ig)}\bigg)^{-\frac{1}{2}}\bigg)^{n_i+m_i} \\
	& = (-1)^{(\alpha \overline{M},\alpha)+(\beta \overline M, \beta) +
		(\alpha \overline M,\beta) - (\beta \overline{M},\alpha)}\prod_{i}\bigg(\eta_g(v_i)\bigg(\frac{\epsilon(v_i,v_i)}{\epsilon(v_ig,v_ig)}\bigg)^{-\frac{1}{2}}\bigg)^{n_i+m_i}\label{ln:eta_pf3} \\
	& = (-1)^{(\alpha \overline{M},\alpha)+(\beta \overline M, \beta) +
		(\alpha \overline M,\beta) + (\beta \overline{M},\alpha)}\prod_{i}\bigg(\eta_g(v_i)\bigg(\frac{\epsilon(v_i,v_i)}{\epsilon(v_ig,v_ig)}\bigg)^{-\frac{1}{2}}\bigg)^{n_i+m_i} \\
	& = (-1)^{((\alpha + \beta) \overline M, (\alpha + \beta))}\prod_{i}\bigg(\eta_g(v_i)\bigg(\frac{\epsilon(v_i,v_i)}{\epsilon(v_ig,v_ig)}\bigg)^{-\frac{1}{2}}\bigg)^{n_i+m_i} \\
	& = \eta_g(\alpha + \beta),
	\end{align}
	where we write $M = \overline{M} - \overline{M}^{\mathrm{T}}$ in line \ref{ln:eta_pf1} and flip a sign in \ref{ln:eta_pf3}. Note that antisymmetry of $M$ implies that
	\begin{equation}
	\frac{\epsilon(v_i,v_i)}{\epsilon(v_ig,v_ig)} = 1.
	\end{equation}
\end{proof}

\subsection{Multiplication of lifted elements}\label{app:liftmult}

Let $g,h \in \mathrm{Aut}(L)$ be lattice automorphisms, and $\hat g, \hat h$ their lifts defined in section~\ref{app:lift},
\be\label{eq:co_s}
\hat{h}\eL_{\alpha} = \eta_{h}(\alpha)^{-1}\eL_{\alpha h}\ , \qquad
\hat{g}\eL_{\alpha} = \eta_{g}(\alpha)^{-1}\eL_{\alpha g}\ .
\ee

\begin{lem}
The multiplication in $\mathrm{Aut}(\hat L)$ is given by
\begin{equation}
	\hat g \hat h = \widehat{gh}s(g,h),
\end{equation}
where
\begin{equation}
	s(g,h) = \frac{\eta_{h}(\cdot)^{-1}\eta_{g}(h\cdot)^{-1}}{\eta_{gh}(\cdot)^{-1}}
\end{equation}
is a map $\mathrm{Aut}(L) \times \mathrm{Aut}(L) \to \mathrm{Hom}(L,\C^*)$ that satisfies the $2$-cocycle condition.
\end{lem}
\begin{proof}
	We calculate
	\be
	\hat{g}\hat{h}\eL_{\alpha} = \eta_{h}(\alpha)^{-1}\eta_{g}(h\alpha)^{-1}\eL_{\alpha hg}
	\ee
	and
	\be
	\widehat{gh}\eL_{\alpha} = \eta_{gh}(\alpha)^{-1}\eL_{\alpha hg}.
	\ee
	If follows that
	\be
	\hat{g}\hat{h} = \widehat{gh} \frac{\eta_{h}(\cdot)^{-1}\eta_{g}(h\cdot)^{-1}}{\eta_{gh}(\cdot)^{-1}}.
	\ee
 The cocycle property can be verified directly.
\end{proof}
Note in particular that $s(g,h)$ is not central in $\mathrm{Aut}(\hat{L})$.

\subsection{Standard Lifts}\label{app:standardlift}
Let $g \in \mathrm{Aut(L)}$ and let be $L_g \subset L$ be the sublattice of $g$-invariant elements. By definition (\ref{eq:eta_def}) the restriction $\eta_g|_{L_g}$ to the fixed-point lattice is a homomorphism. Hence for any $\xi_g \in\mathrm{Hom}(L,\C^*)$ such that $\xi_g|_{L_g} = \eta_g|_{L_g}^{-1}$ the lift $\xi_g(\cdot)\eta_g(\cdot)$ is a standard lift. In particular, note that by lemma \ref{lem:smith} any basis of $L_g$ can be extended to a basis of $L$ and hence such a $\xi_g$ exists.

\subsection{Lifting for $\Z_q \rtimes_{\phi} \Z_p$}\label{app:liftZqZp}

\begin{lem}
	As a finitely generated group $\Z_q \rtimes_{\phi} \Z_p$ is defined by the relations
	\begin{align}
	g^p & =  \id\label{rel:order_p} \\
	h^q & =  \id \label{rel:order_q}\\
	ghg^{-1} & = h^{\phi}\label{rel:comm}.
	\end{align}
\end{lem}
We want to construct a splitting lift $\tilde \eta_g, \tilde \eta_h$ such that $\hat g$, $\hat h$ still satisfy (\ref{rel:order_p}) through (\ref{rel:comm}) and thus generate the same group $\iota(G) = \Z_q \rtimes_{\phi} \Z_p$.

For simplicity we assume that we are given standard lifts $\eta_g$ and $\eta_h$ without order doubling, that is that
the relations \ref{rel:order_p} and \ref{rel:order_q} are automatically satisfied. It is straightforward to generalize our construction otherwise.
Hence for given standard lifts $\eta_g$ and $\eta_h$ define the homomorphism
\begin{equation}\label{fdef}
f(\alpha) := \eta_g(\alpha g^{-1}h)\overline{\eta}_g(\alpha g^{-1})\eta_h(\alpha g^{-1})\prod_{i=0}^{\phi-1}\overline{\eta}_h(\alpha h^i).
\end{equation}
We then have the following lemma:
\begin{lem}
	The relation $\gl^{-1}\hl\gl  = \hl^{\phi}$ is equivalent to $\eta_{g,h}$ satisfying
	\begin{equation}\label{eq:lift}
	f(\alpha)= 1\ .
	\end{equation}
\end{lem}
\begin{proof}
	Since $\gl$ acts on lattice states as
	\begin{equation}
	\gl \mathfrak e_{\alpha} = \eta_g(\alpha)^{-1}\mathfrak e_{\alpha g}\ ,
	\end{equation}
	we have
	\begin{align}\label{liftpower}
	\gl^n\mathfrak e_{\alpha}  = \eta_{g^n}(\alpha)^{-1} e_{\alpha g^n} \ =  \prod_{i=0}^{n-1}\eta_g(\alpha g^i)^{-1} \mathfrak e_{\alpha g^n}
	\end{align}
	and
	\begin{equation}\label{liftinverse}
	\hat g^{-1}\mathfrak e_{\alpha} = \overline{\eta}_g(\alpha g^{-1})^{-1} \mathfrak e_{\alpha g^{-1}}\ .
	\end{equation}
	Using (\ref{liftpower}) and (\ref{liftinverse}) and the fact that $g^{-1}hgh^{-\phi}=e$, a straightforward computation gives
\be
 \hat h^{-\phi} \hat g \hat h\hat g^{-1}\mathfrak e_{\alpha}
=\overline \eta_g(\alpha g^{-1})\eta_h(\alpha g^{-1})\eta_g(\alpha g^{-1}h)\prod_{i=0}^{\phi-1}\overline\eta_h(\alpha h^i)\mathfrak e_{\alpha}\ ,
\ee
which establishes (\ref{eq:lift}).

\end{proof}
In general of course our given standard lifts $\eta_g$ and $\eta_h$ will not satisfy (\ref{eq:lift}). We therefore want to construct homomorphisms $\xi_g$ and $\xi_h$ such that the new lifts $\tilde \eta_g = \eta_g\xi_g$ and $\tilde \eta_h = \eta_h\xi_h$ do satisfy Equation \ref{eq:lift}. In order for the new $\tilde \eta_g$ and $\tilde \eta_h$ to still be standard lifts we demand that $\xi_g(L_g) = \xi_h(L_h) = 1$.
To satisfy (\ref{eq:lift}), $\xi_g$ and $\xi_h$ must satisfy
\begin{equation}\label{eq:def_xi}
\xi_g(\alpha g^{-1}(h-1))\xi_h(\alpha (g^{-1}-\sum_{i=0}^{\phi-1}h^i)) = f(\alpha).
\end{equation}
We now want to construct such homomorphisms $\xi_h,\xi_g$. To do this, we will want to use the Smith formal form of the quotient $\Lambda_h :=L/L_h$. Let us first establish that the action of $g$ is well defined on this quotient:
\begin{lem}\label{lem:Lh}
$gL_h \subset L_h$
\end{lem}
\begin{proof}
The projector $\pi_h = \sum_{i=0}^{q-1} h^i$ onto the $h$ invariant subspace commutes with $g$,
\be
g \pi_h g^{-1} = \pi_h\ .
\ee
This follows from the fact that $\Z_q$ is a normal subgroup of $G$, such that $g \Z_q g^{-1} = \Z_q$. It follows that $gL_h \subset L_h$.
\end{proof}

\begin{cor}
	$f(\alpha) = 1$ for all $\alpha \in L_h$.
\end{cor}
\begin{proof}
	Follows from (\ref{fdef}), the fact that $\eta_h$ is a standard lift, and Lemma~\ref{lem:Lh}.
\end{proof}
We want to define $\xi_g,\xi_h$ using the following basis of $L$:
\begin{lem}\label{thm:decomp}
	There exist two bases $\{\alpha_1, \ldots, \alpha_k, \gamma_1, \ldots, \gamma_l,\epsilon_1, \ldots, \epsilon_m\}$  and $\{\beta_1, \ldots, \beta_k, \delta_1, \ldots, \delta_l,\epsilon_1, \ldots, \epsilon_m\}$ of $L$ together with positive integers $s_i$
	such that
	\bea
	&&\alpha_i (g^{-1}-\sum_{i=0}^{\phi-1}h^i) - s_i \beta_i \in L_h \label{decomp1} \\
	&&\gamma_j(g^{-1}-\sum_{i=0}^{\phi-1}h^i) \in L_h \label{decomp2} \\
	&&\epsilon_j(g^{-1}-\sum_{i=0}^{\phi-1}h^i) \in L_h\ .\label{decomp3}
	\eea
\end{lem}
\begin{proof}
	Let $L = \Lambda_h + L_h$ be a decomposition into the fixed-point sublattice $L_h$ and the primitive sublattice $\Lambda_h$. Pick $\{\epsilon_j\}$ to be a basis of $L_h$. Lemma~\ref{lem:Lh} implies (\ref{decomp3}), and also establishes that $g^{-1}-\sum_{i=0}^{\phi-1}h^i$ acts on the quotient $L/L_h$.
	We can therefore use the Smith normal form for $L/L_h$, meaning that there exist bases $\{a_1, \ldots, a_k, c_1, \ldots, c_l\}$ and $\{b_1, \ldots, b_k, d_1, \ldots, d_l\}$ of $L/L_h$ such that
	\begin{equation}
	a_i(g^{-1}-\sum_{i=0}^{\phi-1}h^i) = s_i b_i
	\end{equation}
	and
	\begin{equation}
	c_j(g^{-1}-\sum_{i=0}^{\phi-1}h^i) = 0,
	\end{equation}
	where $\{s_1, \ldots, s_k\}$ are the elementary divisors of $g^{-1}-\sum_{i=0}^{\phi-1}h^i$.
	Using the isomorphism $\Lambda_h \cong L/L_h$ as free $\Z$-modules, we can lift the above bases of $L/L_h$ to bases $\{\alpha_1, \ldots, \alpha_k, \gamma_1, \ldots, \gamma_l\}$  and $\{\beta_1, \ldots, \beta_k, \delta_1, \ldots, \delta_l\}$ of $\Lambda_h$. (\ref{decomp1}) and (\ref{decomp2}) are then automatically satisfied.
\end{proof}

Using this basis, we can now define the $\xi_g,\xi_h$. It turns out that there is a potential obstruction having to do with the value of $f$ on the $\gamma_i$:

\begin{thm}\label{ZqZplift}
	Let $f$ satisfy $f(\gamma_i) = 1$, $i = 1, \ldots, l$. Then a solution to Equation \ref{eq:def_xi} is given by
	\begin{equation}
	\xi_g(\alpha) = 1, \text{ for all } \alpha \in L
	\end{equation}
	and
	\begin{equation}
	\xi_h(\beta_i) = f(\alpha_i)^{\frac{1}{s_i}}\ , \quad \xi_h(\delta_j) = 1\ , \quad	\xi_h(\epsilon_j) = 1 \ .
	\end{equation}
\end{thm}
\begin{proof}
	We prove this by evaluating (\ref{eq:def_xi}) on the basis $\{\alpha_i,\gamma_j,\epsilon_j\}$. Note that the first factor in (\ref{eq:def_xi}) vanishes because $\xi_g=1$.
	We immediately have
		\be
		f(\epsilon_j)  = 1 =  \xi_h(\epsilon_j(g^{-1}-\sum_{i=0}^{\phi-1}h^i))\ .
		\ee
	By (\ref{decomp1}) we have
	\be
	f(\alpha_i)  =  \xi_h(s_i \beta_i)  = \xi_h(\alpha_i(g^{-1}-\sum_{i=0}^{\phi-1}h^i))
	\ee
	and by (\ref{decomp2})
	\be
f(\gamma_i)  = 1  = \xi_h(\gamma_i(g^{-1}-\sum_{i=0}^{\phi-1}h^i))
\ee
\end{proof}

Let us now discuss under what conditions we indeed have $f(\gamma_i) = 1$.
Our goal is to replace this condition with a slightly different one, which we can then check efficiently on a case by case basis.

\begin{lem}\label{lem:kers}
	\be
	g^{-1}(h-1)(1-g) = -(g^{-1}-\sum_{i=0}^{\phi-1}h^i)(1-h)	\ ,
	\ee
\end{lem}
\begin{proof}
	Expand both sides and use (\ref{rel:comm}).
\end{proof}

\begin{lem}\label{lem:kerxi}
	For $\alpha \in \mathrm{ker}(g^{-1}(h-1)(1-g))$ we have
	\begin{equation}\label{ob1}
	\xi_g(\alpha g^{-1}(h-1)) = 1
	\end{equation}
	and
	\begin{equation}\label{ob2}
	\xi_h(\alpha (g^{-1}-\sum_{i=0}^{\phi-1}h^i)) = 1.
	\end{equation}
\end{lem}
\begin{proof}
	If $\alpha \in \mathrm{ker}(g^{-1}(h-1)(1-g))$, it follows that $\alpha g^{-1}(h-1)\in L_g$, hence (\ref{ob1}). By Lemma~\ref{lem:kers} we can take $\alpha \in \mathrm{ker}((g^{-1}-\sum_{i=0}^{\phi-1}h^i)(1-h))$, which again implies
  $\alpha(g^{-1}-\sum_{i=0}^{\phi-1}h^i)\in L_h$, from which (\ref{ob2}) follows.
\end{proof}

It follows immediately that
\begin{cor}
	There is a solution to (\ref{eq:def_xi}) if and only if $f(\alpha) = 1$ for all $\alpha \in \mathrm{ker}(g^{-1}(h-1)(1-g))$.
\end{cor}
\begin{proof}
	$\gamma_i \in \mathrm{ker}(g^{-1}(h-1)(1-g))$ by (\ref{decomp2}) and Lemma~\ref{lem:kers}, so that Theorem~\ref{ZqZplift} applies. The converse follows from Lemma~\ref{lem:kerxi}.
\end{proof}
We then check by hand that for all our groups, $f(\alpha)=1$ for all $\alpha \in \mathrm{ker}(g^{-1}(h-1)(1-g))$.

\section{Twisted modules for lattice VOAs}

\subsection{Basis of twisted lattice operators}

\begin{thm}\label{thm:Ugeneral}
	Let $\{\alpha_i\}$ be a basis of $L$ and assume that we have a fixed set $U_{\alpha_i}$ that satisfy (\ref{meq:U_mult}). Then the lattice operators defined by
	\begin{equation}
	U_{(\sum_i n_i \alpha_i)} = \prod_{i<j} \big(\epsilon(\alpha_i,\alpha_j)B(\alpha_i,\alpha_j)\big)^{n_in_j}
	\prod_i U_{\alpha_i}^{n_i}\big(\epsilon(\alpha_i,\alpha_i)B(\alpha_i,\alpha_i)\big)^{\frac{n_i(n_i-1)}{2}},
	\end{equation}
	where we use the shorthand
	\begin{equation}
	\prod_i U_{\alpha}^{n_i} := U_{\alpha_1}^{n_1}U_{\alpha_2}^{n_2}\ldots .
	\end{equation}
	all satisfy (\ref{meq:U_mult}).
\end{thm}
\begin{proof}
	Let $\hat \epsilon(\alpha,\beta) = \epsilon(\alpha,\beta)B(\alpha,\beta)$, such that Equation \ref{meq:U_mult} becomes $U_{\alpha+\beta} = \hat{\epsilon}(\alpha,\beta)U_{\alpha} U_{\beta}$. Then we calculate
	\begin{align}
	U_{(\sum_i n_i \alpha_i)} & U_{(\sum_j m_j \alpha_j)} \hat \epsilon\big(\sum_i n_i \alpha_i,\sum_j m_j \alpha_j\big)  = \\
	& =\prod_{i,j}\hat\epsilon(\alpha_i,\alpha_j)^{n_i m_j} \prod_{i<j} \hat\epsilon(\alpha_i,\alpha_j)^{n_in_j + m_im_j}\prod_{i}\hat\epsilon(\alpha_i,\alpha_i)^{\frac{n_i(n_i-1)+m_i(m_i-1)}{2}}  \prod_i U_{\alpha_i}^{n_i} \prod_j U_{\alpha_j}^{m_j} \\
	& =\prod_{i<j} \hat\epsilon(\alpha_i,\alpha_j)^{n_in_j + m_im_j + n_im_j + n_jm_i} \prod_{i}\hat\epsilon(\alpha_i,\alpha_i)^{\frac{n_i(n_i-1)+m_i(m_i-1)}{2}+n_im_i} \prod_i U_{\alpha_i}^{n_im_i} \\
	& = \prod_{i<j} \hat\epsilon(\alpha_i,\alpha_j)^{(n_i+m_i)(n_j+m_j)} \prod_{i}\hat\epsilon(\alpha_i,\alpha_i)^{\frac{(n_i+m_i)(n_i+m_i-1)}{2}} \prod_i U_{\alpha_i}^{n_im_i},
	\end{align}
	where we used the identity
	\begin{equation}
	\prod_i U_{\alpha_i}^{n_i} \prod_j U_{\alpha_j}^{m_j} = \prod_i U_{\alpha_i}^{n_im_i} \prod_{i>j}\bigg(\frac{\hat\epsilon(\alpha_j,\alpha_i)}{\hat\epsilon(\alpha_i,\alpha_j)}\bigg)^{n_im_j}.
	\end{equation}
\end{proof}

\subsection{Twist compatibility}\label{app:twist}

\begin{lem}\label{cor:tw_comp}
	The twist compatibility condition \ref{meq:U_tw} is equivalent to
	\begin{equation}\label{eq:U_tw3}
U_{\alpha(1-g)} = \eta(\alpha)^{-1}\epsilon(\alpha(1-g),\alpha g)
B(\alpha (1-g),\alpha g)^{-1}\exp(2\pi i b_{\alpha})\exp(-2 \pi i \alpha_{(0)}).
\end{equation}
\end{lem}
\begin{proof}
Straightforward computation.
\end{proof}

\begin{lem}\label{cor:Ugeneraltwist}
	If the operators $U_{\alpha}, U_\beta$ satisfy the twist compatibility condition (\ref{meq:U_mult}), then so does $U_{\alpha+\beta}$.
\end{lem}
\begin{proof}
We prove this by showing that  $U_{(\alpha+\beta)(1-g)}$ satisfies  (\ref{eq:U_tw3}).
	Let $\tilde \epsilon(.,.) = \epsilon(.,.)B(.,.)^{-1}$.We calculate
	\begin{align}
	U_{(\alpha+\beta)(1-g)} & =
	\tilde\epsilon(\alpha(1-g),\beta(1-g))^{-1}U_{\alpha(1-g)}U_{\beta(1-g)} \\
	& = \eta(\alpha)^{-1}\eta(\beta)^{-1} \frac{\tilde \epsilon(\alpha(1-g),\alpha g)\tilde \epsilon(\beta(1-g),\beta g)
	}{\tilde \epsilon(\alpha(1-g),\beta(1-g))} \\
	& \quad\quad \exp(2\pi i (b_{\alpha}+b_{\beta}))\exp(-2\pi i (\alpha_{(0)} + \beta_{(0)})) \\
	& = \eta(\alpha+\beta)^{-1}\tilde \epsilon(\beta(1-g),\alpha g)\tilde \epsilon(\alpha(1-g),\alpha g) \\
	& \quad \quad \tilde \epsilon(\beta(1-g),\beta g)\tilde \epsilon(\alpha(1-g),\beta g)\exp(2\pi i (b_{\alpha}+b_{\beta}+\langle \alpha_0|\beta_0 \rangle))\exp(-2\pi i (\alpha+ \beta)_{(0)})\label{ln:cor_comm}   \\
	& = \eta(\alpha+\beta)^{-1}\tilde \epsilon((\alpha+\beta)(1-g),(\alpha+\beta) g)\exp(2\pi i b_{\alpha+\beta})\exp(-2\pi i (\alpha+ \beta)_{(0)}).
	\end{align}
\end{proof}

\subsection{Defect representation $\Omega$}

\begin{thm}\label{thm:lambdabasis}
	Let $\{\alpha_i\}$ be a basis of $L^{\perp}$ such that $s_i \alpha_i \in L(1-g)$ for all $i$ and $\eta(\tilde \alpha_i) = 1$, where $\tilde \alpha_i(1-g) = s_i \alpha_i $ and let $\lambda$ satisfy Equations \ref{meq:lambda_mult} and \ref{meq:lambda_tw2}. Then $\lambda(\alpha_i)$
	is uniquely determined up to an $s_i$-th  root of unity by
	\begin{equation}
	\lambda(\alpha_i) =  \eta(\tilde \alpha_i)^{-\frac{1}{s_i}}\epsilon(\alpha_i,\tilde \alpha_i)B(\alpha_i,\tilde \alpha_i)^{-1}\big(\epsilon(\alpha_i,\alpha_i)B(\alpha_i,\alpha_i)\big)^{\frac{s_i+1}{2}}\exp(\frac{2\pi i}{s_i} b_{\tilde \alpha_i})
	\end{equation}

\end{thm}
\begin{proof}
	By Equation \ref{meq:lambda_mult} we have
	\begin{equation}
	\lambda(s_i \alpha_i) = \lambda(\alpha_i)^{s_i} (\epsilon(\alpha_i,\alpha_i)B(\alpha_i,\alpha_i))^{\frac{s_i(s_i-1)}{2}},
	\end{equation}
	 where we may choose a basis such that $\chi_{\alpha_i}(x_{\alpha_i})$ and by Equation \ref{meq:lambda_tw2}
	\begin{align}
	\lambda(s_i \alpha_i) & = \lambda(\tilde \alpha_i(1-g)) \\
	& = \eta(\tilde \alpha_i)^{-1}\epsilon(\tilde \alpha_i(1-g),\tilde\alpha_i g)
	B(\tilde \alpha_i (1-g),\tilde \alpha_i g)^{-1}\exp(2\pi i b_{\tilde \alpha_i}) \\
	& = \eta(\tilde \alpha_i)^{-1}\epsilon(s_i\alpha_i,\tilde\alpha_i g)B(s_i\alpha_i,\tilde\alpha_i g)^{-1}\exp(2\pi i b_{\tilde \alpha_i}) \\
	& = \eta(\tilde \alpha_i)^{-1}\epsilon(s_i\alpha_i,\tilde\alpha_i - s_i\alpha_i)B(s_i\alpha_i,\tilde\alpha_i - s_i\alpha_i)^{-1}\exp(2\pi i b_{\tilde \alpha_i}) \\
	& = \eta(\tilde \alpha_i)^{-1}\big(\epsilon(\alpha_i,\tilde\alpha_i)B(\alpha_i,\tilde\alpha_i)^{-1}\big)^{s_i}\big(\epsilon(\alpha_i,\alpha_i)B(\alpha_i,\alpha_i)\big)^{s_i^2}\exp(2\pi i b_{\tilde \alpha_i})
	\end{align}
	Hence it follows that
	\begin{equation}
	\lambda(\alpha_i)^{s_i} = \eta(\tilde \alpha_i)^{-1}\big(\epsilon(\alpha_i,\tilde\alpha_i)B(\alpha_i,\tilde\alpha_i)^{-1}\big)^{s_i}\big(\epsilon(\alpha_i,\alpha_i)B(\alpha_i,\alpha_i)\big)^{\frac{s_i(s_i+1)}{2}}\exp(2\pi i b_{\tilde \alpha_i}),
	\end{equation}
	so that in particular $\lambda(\alpha_i)$ is determined up to an $s_i$-th root of unity.
\end{proof}

\medskip

\bibliographystyle{alpha}
%\bibliographystyle{../../ytphys}
%\bibliography{../../refmain}
% \bibliographystyle{../ytphys}
 \bibliography{../refmain}
%\bibliographystyle{./ytphys}
%\bibliography{./refmain}

\end{document}